\documentclass[a4paper, twoside, 11pt, english]{article}

\usepackage[T1]{fontenc}
\usepackage[utf8]{inputenc}
\usepackage{etoolbox}
\usepackage{amsmath,amsthm,amssymb,stmaryrd}
\usepackage{mathtools}
\usepackage[ttscale=.875]{libertine}
\usepackage[libertine]{newtxmath}
\usepackage[numbers,sort&compress]{natbib}
\usepackage{comment}
\usepackage{ifthen}
\usepackage{algpseudocode}
\usepackage{algorithm2e}
\usepackage{multicol}
\usepackage{fancyhdr, titlesec, url, enumerate, microtype,setspace}
\usepackage[all,knot,poly]{xy}
\usepackage{tikz}
\usetikzlibrary{decorations.pathreplacing}

\usepackage{tikz-qtree,varwidth}
\usepackage{youngtab}
\usepackage{ytableau}
\usepackage[english]{babel}
\usepackage{hyperref}

\usetikzlibrary{calc}
\CompileMatrices

\newcommand{\periodafter}[1]{\ifstrempty{#1}{}{#1.}}
\titleformat{\section}[block]{\scshape\filcenter\LARGE}{\thesection.}{.5em}{}
\titleformat{\subsection}[block]{\bfseries\filcenter\large}{\thesubsection.}{.5em}{\medskip}
\titleformat{\subsubsection}[runin]{\bfseries}{\thesubsubsection.}{.5em}{\periodafter}
\titlespacing{\subsubsection}{0pt}{\topsep}{.5em}

\titleformat{\section}[block]{\scshape\filcenter\Large}{\thesection.}{.5em}{}
\titleformat{\subsection}[runin]{\bfseries}{\thesubsection.}{.5em}{}[.]
\titleformat{\subsubsection}[runin]{\bfseries}{\thesubsubsection.}{.5em}{}[.]
\titlespacing{\subsubsection}{0pt}{10pt}{.5em}

\theoremstyle{ntheorem}
  	\newtheorem{theorem}[subsubsection]{Theorem}
  	\newtheorem{proposition}[subsubsection]{Proposition}
	\newtheorem{lemma}[subsubsection]{Lemma}
  	\newtheorem{corollary}[subsubsection]{Corollary}

\theoremstyle{definition}

	\newtheorem{example}[subsubsection]{Example}
	
	\newtheorem{Algorithm}[subsubsection]{Algorithm}
	\newtheorem{property}[subsubsection]{Property}
	
	\makeatletter
\def\@equationname{equation}

\makeatother

\fancypagestyle{plain}{
\fancyhf{}\fancyfoot[LC,RC]{}
\fancyfoot[LE,RO]{$\thepage$}

}

\newcount\hh
\newcount\mm
\mm=\time
\hh=\time
\divide\hh by 60
\divide\mm by 60
\multiply\mm by 60
\mm=-\mm
\advance\mm by \time
\def\hhmm{\number\hh:\ifnum\mm<10{}0\fi\number\mm}

\UseTips
\SelectTips{eu}{11}

\newdir{ >}{{}*!/-10pt/@{>}}
\newdir{ -}{{}*!/-10pt/@{}}
\newdir{> }{{}*!/+10pt/@{>}}

\makeatletter

\xyletcsnamecsname@{dir4{}}{dir{}}
\xydefcsname@{dir4{-}}{\line@ \quadruple@\xydashh@}
\xydefcsname@{dir4{.}}{\point@ \quadruple@\xydashh@}
\xydefcsname@{dir4{~}}{\squiggle@ \quadruple@\xybsqlh@}
\xydefcsname@{dir4{>}}{\Tttip@}
\xydefcsname@{dir4{<}}{\reverseDirection@\Tttip@}

\xydef@\quadruple@#1{%
	\edef\Drop@@{%
		\dimen@=#1\relax
		\dimen@=.5\dimen@
		\A@=-\sinDirection\dimen@
		\B@=\cosDirection\dimen@
		\setboxz@h{%
			\setbox2=\hbox{\kern3\A@\raise3\B@\copy\z@}%
			\dp2=\z@ \ht2=\z@ \wd2=\z@ \box2
			\setbox2=\hbox{\kern\A@\raise\B@\copy\z@}%
			\dp2=\z@ \ht2=\z@ \wd2=\z@ \box2
			\setbox2=\hbox{\kern-\A@\raise-\B@\copy\z@}%
			\dp2=\z@ \ht2=\z@ \wd2=\z@ \box2
			\setbox2=\hbox{\kern-3\A@\raise-3\B@ \noexpand\boxz@}%
			\dp2=\z@ \ht2=\z@ \wd2=\z@ \box2
		}%
		\ht\z@=\z@ \dp\z@=\z@ \wd\z@=\z@ \noexpand\styledboxz@
	}%
}

\xydef@\Tttip@{\kern2pt \vrule height2pt depth2pt width\z@
	\Tttip@@ \kern2pt \egroup
	\U@c=0pt \D@c=0pt \L@c=0pt \R@c=0pt \Edge@c={\circleEdge}%
	\def\Leftness@{.5}\def\Upness@{.5}%
	\def\Drop@@{\styledboxz@}\def\Connect@@{\straight@{\dottedSpread@\jot}}}
	
\xydef@\Tttip@@{%
	\dimen@=.25\dimen@
 	\B@=\cosDirection\dimen@
	\setboxz@h\bgroup\reverseDirection@\line@ \wdz@=\z@ \ht\z@=\z@ \dp\z@=\z@
	{\vDirection@(1,-1)\xydashl@ \xyatipfont\char\DirectionChar}%
	{\vDirection@(1,+1)\xydashl@ \xybtipfont\char\DirectionChar}%
}

\xydef@\ar@form{
	\ifx \space@\next \expandafter\DN@\space{\xyFN@\ar@form}%
	\else\ifx ^\next \DN@ ^{\xyFN@\ar@style}\edef\arvariant@@{\string^}%
	\else\ifx _\next \DN@ _{\xyFN@\ar@style}\edef\arvariant@@{\string_}%
	\else\ifx 0\next \DN@ 0{\xyFN@\ar@style}\def\arvariant@@{0}%
	\else\ifx 1\next \DN@ 1{\xyFN@\ar@style}\def\arvariant@@{1}%
	\else\ifx 2\next \DN@ 2{\xyFN@\ar@style}\def\arvariant@@{2}%
	\else\ifx 3\next \DN@ 3{\xyFN@\ar@style}\def\arvariant@@{3}%
	\else\ifx 4\next \DN@ 4{\xyFN@\ar@style}\def\arvariant@@{4}%
	\else\ifx \bgroup\next \let\next@=\ar@style
	\else\ifx [\next \DN@[##1]{\ar@modifiers{[##1]}}
	\else\ifx *\next \DN@ *{\ar@modifiers}%
	\else\addLT@\ifx\next \let\next@=\ar@slide
	\else\ifx /\next \let\next@=\ar@curveslash
	\else\ifx (\next \let\next@=\ar@curveinout 
	\else\addRQ@\ifx\next \addRQ@\DN@{\ar@curve@}%
	\else\addLQ@\ifx\next \addLQ@\DN@{\xyFN@\ar@curve}%
	\else\addDASH@\ifx\next \addDASH@\DN@{\defarstem@-\xyFN@\ar@}%
	\else\addEQ@\ifx\next \addEQ@\DN@{\def\arvariant@@{2}\defarstem@-\xyFN@\ar@}%
	\else\addDOT@\ifx\next \addDOT@\DN@{\defarstem@.\xyFN@\ar@}%
	\else\ifx :\next \DN@:{\def\arvariant@@{2}\defarstem@.\xyFN@\ar@}%
	\else\ifx ~\next \DN@~{\defarstem@~\xyFN@\ar@}%
	\else\ifx !\next \DN@!{\dasharstem@\xyFN@\ar@}%
	\else\ifx ?\next \DN@?{\ar@upsidedown\xyFN@\ar@}%
	\else \let\next@=\ar@error
	\fi\fi\fi\fi\fi\fi\fi\fi\fi\fi\fi\fi\fi\fi\fi\fi\fi\fi\fi\fi\fi\fi\fi \next@}

\makeatother

\newcommand{\qfl}{\xymatrix@1@C=10pt{\ar@4 [r] &}}

\DeclareMathOperator{\LR}{LR}
\DeclareMathOperator{\U}{\mathcal{U}}
\DeclareMathOperator{\Tc}{\mathcal{T}}
\DeclareMathOperator{\Si}{\mathcal{S}}

\DeclareMathOperator{\Rec}{Rect}
\DeclareMathOperator{\YoungST}{Ys}
\DeclareMathOperator{\YoungT}{Yt}
\DeclareMathOperator{\SkewT}{St}
\DeclareMathOperator{\RSK}{sRSK}
\DeclareMathOperator{\Ba}{Ba}

\renewcommand{\phi}{\varphi}
\renewcommand{\epsilon}{\varepsilon}

\newcommand{\Zb}{\mathbb{Z}}

\renewcommand{\Pr}{\EuScript{P}}

\newcommand{\Yr}{\EuScript{Y}}

\newcommand{\M}{\mathbf{M}}

\renewcommand{\leq}{\leqslant}
\renewcommand{\geq}{\geqslant}

\newcommand{\insl}[1]{\rightsquigarrow_{#1}}
\newcommand{\insr}[1]{\;\raisebox{0.1em}{{\rotatebox[origin=c]{180}{$\rightsquigarrow$}}}_{#1}\;}

\definecolor{cyan}{RGB}{175,238,238} 
\definecolor{Red}{rgb}{0.96,0.17,0.20}
\definecolor{RedD}{rgb}{0.57, 0.0, 0.04}
\definecolor{Green}{rgb}{0.0,1.0,0.0}
\definecolor{GreenL}{rgb}{0.0,1.0,0.0} 
\definecolor{GreenD}{rgb}{0.64,0.76,0.68} 
\definecolor{Yellow}{rgb}{1.0,1.0,0.19}
\definecolor{YellowD}{rgb}{0.93, 0.84, 0.25}
\definecolor{Blue}{rgb}{0.0,1.0,1.0}
\definecolor{BlueD}{rgb}{0.63,0.79,0.95}
\definecolor{vert}{rgb}{0,0.45,0}
\definecolor{bazaar}{rgb}{0.6, 0.47, 0.48}
\definecolor{bronze}{rgb}{0.8, 0.5, 0.2}
\definecolor{darkspringgreen}{rgb}{0.09, 0.45, 0.27}

\def\Pr{\mathcal{P}}

\def\Yr{\mathcal{Y}}
\makeatletter
\def\blfootnote{\xdef\@thefnmark{}\@footnotetext}
\makeatother
\def\R{\mathbf{R}}
\def\P{\mathbf{P}}
\def\T{\mathbf{T}}
\def\Q{\mathbf{Q}}

\begin{document}
\thispagestyle{empty}

\begin{center}

\begin{doublespace}
\begin{huge}
{\scshape A super Robinson--Schensted--Knuth correspondence  with symmetry and the super  Littlewood--Richardson rule}
\end{huge}

\bigskip
\hrule height 1.5pt 
\bigskip

\begin{Large}
{\scshape Nohra Hage }
\end{Large}

\end{doublespace}

\vspace{1cm} 

\begin{small}\begin{minipage}{14cm}
\noindent\textbf{Abstract -- } 
The Robinson--Schensted--Knuth (RSK) correspondence is a bijective correspondence between two-rowed arrays of non-negative integers and pairs of same-shape semistandard  tableaux.
This correspondence satisfies the symmetry property, that is, exchanging the rows of a two-rowed array is equivalent to exchanging the positions of the corresponding  pair of semistandard tableaux.
In this article, we introduce a super-analogue of the RSK correspondence for super tableaux over a signed alphabet using a super version of Schensted's insertion algorithms. 
We give a geometrical interpretation of the super-RSK correspondence via a matrix-ball construction, showing the symmetry property in complete generality.
Finally, we deduce a combinatorial version of the super Littlewood--Richardson rule for super Schur functions over a finite signed alphabet.

\medskip


\medskip

\smallskip\noindent\textbf{M.S.C. 2010 -- Primary:} 05E99. \textbf{Secondary:} 20M99, 05A19, 	16T99.
\end{minipage}\end{small}
\end{center}

\begin{small}\begin{minipage}{12cm}
\renewcommand{\contentsname}{}
\setcounter{tocdepth}{1}
\tableofcontents
\end{minipage}
\end{small}

\tikzset{every tree node/.style={minimum width=1em,draw,circle},
         blank/.style={draw=none},
         edge from parent/.style=
         {draw,edge from parent path={(\tikzparentnode) -- (\tikzchildnode)}},
         level distance=0.8cm}
\vspace{0.5cm}

\section{Introduction}
Schensted introduced in~\cite {Schensted61} a bijection between permutations over the totally ordered alphabet~$[n]:=\{1<\ldots<n\}$ and pairs of same-shape standard Young tableaux over~$[n]$ in order to compute the length of the longest decreasing  subsequence of a given permutation over~$[n]$. 
This correspondence is described using Schensted's insertion procedure that constructs a first standard Young tableau by successively inserting the elements of the given permutation according to a specific rule, while the second standard Young tableau records the evolution of the shape during the insertion. This correspondence had been also described, in a rather different form, much earlier by Robinson in~\cite{Robinson38} in an attempt to give a first correct proof of the Littlewood--Richardson rule that provides an explicit combinatorial description for expressing a skew Schur function or a product of two Schur functions as a linear combination of Schur functions. This correspondence is then referred to as the \emph{Robinson--Schensted (RS) correspondence}.
Knuth generalized in~\cite{Knuth70} the RS correspondence to a bijection between two-rowed arrays of elements of~$[n]$ and pairs of same-shape semistandard Young tableaux over~$[n]$. 
Knuth's bijection is also described using Schensted's insertion algorithm that constructs a first semistandard Young tableau by successively inserting the elements of the second row of the given two-rowed array from left to right, while the second semistandard Young tableau records the evolution of the shape during this insertion using the elements of the first row of the two-rowed array.  
This bijection is then known as the \emph{Robinson--Schensted--Knuth (RSK) correspondence}.
An essential property of this correspondence is that it satisfies the \emph{symmetry property}, that is, under the RSK correspondence, exchanging the rows of a two-rowed array is equivalent to exchanging the positions of the corresponding  pair of semistandard Young tableaux,~\cite{Knuth70}. This property is also proved by Viennot in~\cite{Viennot77} for the  RS correspondence, and by Fulton in~\cite{Fulton97} for the RSK correspondence, using geometrical interpretations of these correspondences.  Since then, the RSK correspondence has found rich applications in representation theory, algebraic combinatorics and probabilistic combinatorics,~\cite{Kashiwara91, Fulton97, Leeuwen01, OConnell03}, and has found many generalizations on other structures of tableaux,~\cite{SaganStanley90, ShimozonoWhite01, Leeuwen05, BereleRemmel85, BereleRegev87, BonettiSenatoVenezia98, Muth19}.

Superalgebraic structures have served as powerful combinatorial tools in the study of invariant theory and representation theory of superalgebras, as well as in the theory of algebras satisfying polynomial identities,~\cite{BonettiSenatoVenezia98,GrosshansRotaStein87, BereleRemmel85,BereleRegev87,BenkartKangKashiwara20}.
Bonetti, Senato and Venezia introduced in~\cite{BonettiSenatoVenezia98} a super-RSK correspondence on super Young tableaux over a signed alphabet using super Schensted's right and left insertion. However, the symmetry property holds only in special cases under this bijection as shown in~\cite{LaScalaNardozzaSenato06}. A question was to find a super-RSK correspondence
satisfying the symmetry property in complete generality and leading to the classical RSK correspondence as a particular case.
Muth  introduced in~\cite{Muth19} a super-RSK correspondence using Haiman's mixed insertion algorithm on super Young tableaux and proved that this correspondence satisfies the symmetry property in complete generality. However, this correspondence is not related to the super plactic monoid of type A and to the representations of the general linear Lie superalgebra, and  then it does not yield  a combinatorial description of the super Littlewood--Richardson rule. It is worth noting that Haiman's mixed insertion algorithm is used to define the shifted plactic monoid and allows to give a combinatorial version of the shifted Littlewood--Richardson rule for shifted tableaux,~\cite{Serrano10}.
We introduce in this article a super version of the RSK correspondence on super Young tableaux that satisfies the symmetry property in complete generality, using super Schensted's insertion algorithms and  the super plactic monoid of type A.
We deduce combinatorial descriptions of the super Littlewood--Richardson rule for super Young tableaux over a finite signed alphabet.

\subsubsection*{Super plactic monoids, insertion and taquin}
A \emph{signed alphabet} is a finite or countable totally ordered set~$\Si$ which is the disjoint union of two subsets~$\Si_0$ and~$\Si_1$. 
The \emph{super plactic monoid}~$\P(\Si)$ over a signed alphabet~$\Si$,  is the quotient of the free monoid~$\Si^\ast$ over~$\Si$ by the congruence relation~$\sim_{\P(\Si)}$ generated by the following family of \emph{super Knuth relations},~\cite{LaScalaNardozzaSenato06}:
\[
\begin{array}{rl}
xzy=zxy,\;\text{ with }\; x=y \;\text{ only if }\; y\in \Si_0 \;\text{ and } \; y=z\; \text{ only if }\; y \in \Si_1,\\
yxz=yzx,\;\text{ with }\; x=y \;\text{ only if }\;  y\in \Si_1\;\text{ and } \; y=z\; \text{ only if }\; y\in \Si_0,
\end{array}
\]
for any~$x\leq y\leq z$ of elements of~$\Si$. 
When~$\Si=\Si_0=[n]$, we recover the \emph{plactic monoid of type~A} introduced  by Lascoux and Sch\"{u}tzenberger in~\cite{LascouxSchutsenberger81}, following the works of Schensted,~\cite{Schensted61}, and Knuth,~\cite{Knuth70}, on the RSK correspondence.  Plactic monoids have found several applications in algebraic combinatorics, representation theory, probabilistic combinatorics and rewriting theory,~\cite{Lothaire02,Fulton97,OConnell03, KubatOkninski14, BokutChenChenLi15, CainGrayMalheiro15, CainGrayMalheiro19, HageMalbos17, HageMalbos22, Hage15, Hage2021Super}. 
When~$\Si=\{\overline{m}<\ldots <\overline{1}<1<\ldots<n\}$, we recover the reverse of the super Knuth relations obtained in~\cite{BenkartKangKashiwara20} where the super plactic congruence is described using Kashiwara's theory of crystal bases for the representations of the general linear Lie superalgebra~$\mathfrak{gl}_{m,n}$.
It is worth noting that the super plactic monoid also appeared  in~\cite{LodayPovov08} in the study of  the parastatistics algebra.

A \emph{partition} of a positive integer~$n$ is a weakly decreasing sequence~$\lambda=(\lambda_1,\ldots,\lambda_k)\in\Zb_{>0}^k$ such that~$ \sum\lambda_i = n$. 
The \emph{Young diagram of a partition}~$\lambda=(\lambda_1,\ldots,\lambda_k)$ is the set~$\Yr(\lambda)$ of pairs~$(i,j)$ such that~$1\leq i\leq k$ and~$1\leq j\leq \lambda_i$,
that can be represented by a diagram by drawing a box for each pair~$(i,j)$. The transposed diagram~$\{(j,i)\;\big|\; (i,j)\in\Yr(\lambda)\}$ defines another partition~$\widetilde{\lambda}$,  called the \emph{conjugate partition} of~$\lambda$.
A  \emph{super tableau of shape}~$\lambda$ over~$\Si$ is a Young diagram of~$\lambda$ filled with elements of~$\Si$ such that the entries in each row are weakly increasing  allowing the repetition only of elements in~$\Si_0$ and the entries in each column are weakly increasing allowing the repetition only of elements in~$\Si_1$.
Denote by~$\YoungT(\Si)$  the set of all super tableaux over~$\Si$.  Note that, when~$\Si=\Si_0$ and~$\Si=\Si_1$, we recover the notion of row-strict and column-strict tableaux, respectively,~\cite{Fulton97}. 
Denote by~$\R:\YoungT(\Si)\to \Si^\ast$ the \emph{column reading map} that reads a super tableau column-wise from bottom to top and from left to right. 
Super versions of Schensted left and right insertion algorithms are  introduced in~\cite{LaScalaNardozzaSenato06} on super tableaux, and consist in inserting elements of~$\Si$ into super  tableaux  by rows and columns, respectively.  For any word $w=x_1\ldots x_k$  over~$\Si$, a super tableau~$\T(w)$ is computed  by inserting the letters of~$w$ iteratively from left to right using the right insertion~$\insr{}$, and starting from the empty super tableau:
\[
\T(w) 
\;
:=
\;
(\emptyset \insr{} w)
\;
=
\;((\ldots(\emptyset \insr{} x_1)  \insr{} \ldots )\insr{} x_k).
\]
Two words over~$\Si$ are equivalent with respect to~$\sim_{\P(\Si)}$ if and only if they lead to the same super tableau after insertion,~\cite{LaScalaNardozzaSenato06}. This is the \emph{cross-section property} of super tableaux with respect to~$\sim_{\P(\Si)}$. We deduce that the internal product~$\star_{\YoungT(\Si)}$ defined on $\YoungT(\Si)$ by~$t \star_{\YoungT(\Si)} t' :=  (t\insr{} \R(t'))$, 
for all~$t$ and~$t'$ in~$\YoungT(\Si)$, is associative, and
the monoids~$(\YoungT(\Si), \star_{\YoungT(\Si)})$ and~$\P(\Si)$ are isomorphic.

Let~$\lambda$ and~$\mu$ be partitions such that~$\Yr(\mu)$ is contained in~$\Yr(\lambda)$.
A  \emph{super skew tableau of shape}~$\lambda/\mu$  over~$\Si$ is a Young diagram of the following form
\[
\Yr(\lambda/\mu):=\big\{(i,j)\;\big|\; 1\leq i\leq k,\; \mu_i< j\leq \lambda_i \big\}
\]
filled with elements of~$\Si$ such that the entries in each row are weakly increasing  allowing the repetition only of elements in~$\Si_0$ and the entries in each column are weakly increasing  allowing the repetition only of elements in~$\Si_1$.
An \emph{inner corner} of a super skew tableau of shape~$\lambda/\mu$  is a box in~$\Yr(\mu)$ such that the boxes below and to the right are not in~$\Yr(\mu)$, and an \emph{outer corner} is a box such that neither box below or to the right is in~$\Yr(\lambda/\mu)$. 
Sch\"{u}tzenberger introduced in~\cite{Schutzenberger63} the \emph{jeu de taquin procedure} on Young tableaux in order to give a  proof of the Littlewood--Richardson rule using the properties of the plactic monoid of type~A.
We introduced in~\cite{Hage2021Super} the \emph{super jeu taquin procedure}  on super tableaux that consists in moving inner corners from a super skew tableau into outer corners by keeping the rows and the columns weakly increasing until no more inner corners remain in the initial super skew tableau.  We proved that the \emph{rectification} of a super skew tableau~$S$ by the super jeu de taquin is the unique super  tableau whose reading is equivalent to the one of~$S$ with respect to~$\sim_{\P(\Si)}$.
We deduced that the resulting super  tableau does not depend on the order in which we choose inner corners. We also related the super jeu de taquin to the insertion algorithms and we showed how we can insert a super tableau into another one by taquin. This interpretation of the insertion product~$\star_{\YoungT(\Si)}$ by taquin allows us to give in Section~\ref{S:SuperLittelwoodRichardsonRule} a combinatorial description of the super Littlewood--Richardson coefficient. 
Moreover, we introduced in~\cite{Hage2021Super} the super analogue of the \emph{Sch\"{u}tzenberger's  evacuation procedure} which transforms a super tableau~$t$ over a signed alphabet~$\Si$ into an \emph{opposite tableau}~$t^{\mathrm{op}}$, over the opposite alphabet obtained from~$\Si$ by reversing its order. We showed that the super tableaux~$t$ and~$t^{\mathrm{op}}$ have the same shape and that the map~$t\mapsto  t^{\mathrm{op}}$ is an involution on~$\YoungT(\Si)$ that is  compatible with~$\sim_{\P(\Si)}$. 
The super evacuation procedure  allows us to construct in Subsection~\ref{SS:DualRSKCorrespondence} a dual version of the super-RSK correspondence using the left insertion algorithm on super tableaux.

\section*{ Organization and main results of the article}

We begin by recalling in Section~\ref{S:SuperPlacticMonoid} the notions of super Young tableaux, the super plactic monoid and the super jeu de taquin from~\cite{LaScalaNardozzaSenato06, Hage2021Super}.
Moreover, we give a super version of the Robinson--Schensted correspondence for super tableaux over a signed alphabet.

\subsection*{A super-RSK correspondence with symmetry}
We  introduce in Subsection~\ref{SS:SuperRSKCorrespondence} a super version of the RSK correspondence over a signed alphabet. 
Let~$\Si$ and~$\Si'$ be signed alphabets. 
A \emph{signed two-rowed array}~$w$ on~$\Si$ and~$\Si'$ is a $2\times k$ matrix 
\[
w
\;:=\;
\begin{pmatrix}
 x_1&\ldots &x_k\\
 y_1&\ldots &y_k
\end{pmatrix}
\]
with~$x_i$ in~$\Si$ and~$y_i$ in~$\Si'$, for all~$i=1,\ldots, k$, that satisfies Conditions~\eqref{E:Order1} and~\eqref{E:Order2}.
Starting from a signed two-rowed array~$w$ on~$\Si$ and~$\Si'$, Algorithm~\ref{A:RSkalgo} computes a pair of same-shape super tableaux~$(\T(w), \Q(w))$  whose entries are the
ones of the second and the first row of~$w$, respectively.
More precisely, the super tableau~$\T(w)$ is equal to~$\T(y_1\dots y_k)$, and~$\Q(w)$ is the super tableau obtained by successively adding~$x_1,\ldots, x_k$  in the same places as the boxes added when computing~$\T(w)$ starting from an empty super tableau. Moreover, Algorithm~\ref{A:ReverseRSKCorrespodence}
allows us to recover the initial signed two-rowed array~$w$ starting from the same-shape pair of super tableaux~$(\T(w), \Q(w))$. This sets up a one-to-one correspondence between signed two-rowed arrays  and pairs of same-shape super tableaux on~$\Si$ and~$\Si'$, that we denote by
\[
\RSK: w \mapsto (\T(w),\Q(w)).
\]
Hence, we obtain the following first main result of the article:
 
 \begin{quote}
{\bf Theorem \ref{T:SuperRSKCorrespondence}[Super-RSK correspondence].}
\emph{
Let $\Si$ and~$\Si'$ be  signed alphabets. The map~$\RSK$ defines a  one-to-one correspondence  between signed two-rowed arrays and pairs of super
 tableaux on~$\Si$ and~$\Si'$, such that for any signed two-rowed array~$w$ on~$\Si$ and~$\Si'$, we have that $\T(w)$ and~$\Q(w)$ are same-shape super tableaux whose entries are the
ones of the second and the first row of~$w$, respectively.}
\end{quote}
 
This is the super version of the Robinson--Schensted--Knuth correspondence between two-rowed arrays of non-negative integers and pairs of same-shape semistandard tableaux over~$[n]$,~\cite{Knuth70, Fulton97}. More precisely,  when~$\Si=\Si'=\Si_0=\Si'_0=[n]$, we recover the classical RSK correspondence on Young tableaux of type A. 
It is worth noting that, in the particular case, when all the elements of~$\Si_0$  (resp.~$\Si'_0$) are strictly smaller than the ones of~$\Si_1$ (resp.~$\Si'_1$), we recover the super-RSK correspondence on super tableaux introduced by Berele and Remmel  in~\cite{BereleRemmel85}.

Using a matrix-ball construction, we show in Subsection~\ref{SS:SuperMatrixBallConstruction} that  the super-RSK correspondence satisfies the symmetry property in complete generality. 
Let~$w$ be a signed two-rowed array on signed alphabets~$\Si$ and~$\Si'$.
The \emph{inverse} of~$w$, denoted by~$w^{\text{inv}}$, is the  signed two-rowed array on~$\Si'$ and~$\Si$
obtained from~$w$ by \emph{exchanging} the rows of~$w$, and by sorting the new couples on~$\Si'\times\Si$ according to Conditions~\eqref{E:Order1} and~\eqref{E:Order2}. 
The signed two-rowed array~$w$ has \emph{symmetry with respect to the map}~$\RSK$  if it satisfies the following property:
\[
\text{if}~\RSK(w)= (\T(w),\Q(w)) \quad \text{ then } \quad  \RSK(w^{\text{inv}}) = (\Q(w),\T(w)).
\]
A \emph{signed ball array} on signed alphabets~$\Si$ and~$\Si'$ is a rectangular array of balls filled with positive integers, whose rows and columns are indexed with elements of~$\Si$ and~$\Si'$, respectively, allowing the repetition only of elements in~$\Si_1$ and~$\Si'_1$, respectively, and where many balls can occur in the same position. 
We  show in~\ref{SSS:MatrixBallConstruction} how to associate 
to each signed two-rowed array~$w$  on~$\Si$ and~$\Si'$,  a signed ball array, denoted by~$\Ba(w)$, whose rows and columns are indexed with the elements of the first and second row of~$w$, respectively, from the smaller to the bigger one,  and where only the indices from~$\Si_1$ and~$\Si'_1$ are repeated as many times as they appear in~$w$. 
First, we  associate to each couple in~$w$  a ball in an empty signed ball array and then we order and number these balls with positive integers. Secondly, we add new balls to the given signed ball array, we order and number them and we repeat the same procedure until no more balls can be added. This geometrical construction allows us to prove the symmetry property using the symmetry of the resulting signed ball array~$\Ba(w)$. More precisely, we denote by~$\T(\Ba(w))$ and~$\Q(\Ba(w))$  the super tableau obtained from~$\Ba(w)$ whose $k$-th row lists the indices of the leftmost columns and top-most rows, respectively, where each integer number occurs in the new added balls.
Proposition~\ref{P:MatrixBallConstruction} shows that~$(\T(\Ba(w)), \Q(\Ba(w))) =(\T(w), \Q(w))$. 
Since the matrix-ball construction is symmetric in the rows and columns of the resulting signed ball array, we deduce the following result:

 \begin{quote}
{\bf Theorem \ref{T:SymmetryRSK}[Symmetry of the super-RSK correspondence].}
\emph{
Let~$\Si$ and~$\Si'$ be signed alphabets. All signed two-rowed arrays on~$\Si$ and~$\Si'$ have symmetry with respect to the super-RSK correspondence map~$\RSK$.
}
\end{quote}
 
We end Subsection~\ref{SS:DualRSKCorrespondence} by giving a  dual way to construct the pair of super tableaux corresponding to a signed two-rowed array with respect to the map~$\RSK$ using the left insertion algorithm and the super evacuation procedure on super tableaux; this construction extends the classical dual RSK correspondence,~\cite{Knuth70, Fulton97}, to the signed setting.

\subsection*{The super Littlewood--Richardson rule}
 
 We give  in Section~\ref{S:SuperLittelwoodRichardsonRule} a combinatorial interpretation of the super Littlewood--Richardson rule using the super-RSK correspondence. Let~$\lambda$,~$\mu$ and~$\nu$ be partitions such that~$\Yr(\lambda)$ is contained in~$\Yr(\nu)$ and let~$\Si$ be a finite signed alphabet. We show that the number~$c_{\lambda,\mu}^{\nu}$ of ways a given super tableau~$t$ of shape~$\nu$ over~$\Si$ can be written as the product of a super tableau~$t'$ of shape~$\lambda$ and a super tableau~$t''$ of shape~$\mu$ over~$\Si$, does not depend on~$t$, and depends only on the partitions~$\lambda$,~$\mu$  and~$\nu$. Moreover, we show that~$c_{\lambda,\mu}^{\nu}$ is  equal to the number of super skew tableaux of shape~$\nu/\lambda$  whose rectification is a given tableau of shape~$\mu$. Hence, we obtain the following result:

  \begin{quote}
{\bf Theorem \ref{T:SuperLittlewoodRichardsonRule}[The super Littlewood--Richardson rule].}
\emph{
Let~$\lambda$,~$\mu$ and~$\nu$ be partitions such that~$\Yr(\lambda)$ is contained in~$\Yr(\nu)$.  The following identities
\[
S_{\lambda} S_{\mu}\;=\; \underset{\nu}{\sum} c_{\lambda,\mu}^{\nu} S_{\nu}
\qquad
\text{ and }
\qquad
S_{\nu/\lambda}\;=\; \underset{\mu}{\sum} c_{\lambda,\mu}^{\nu} S_{\mu}
\]
hold in the tableau~$\mathbb{Z}$-algebra rising from~$\P(\Si)$, where~$S_{\nu/\lambda}$, $S_{\lambda}$,~$S_{\mu}$ and~$S_{\nu}$ denote respectively the formal sum of all super tableaux of shape~$\nu/\lambda$,~$\lambda$,~$\mu$ and~$\nu$  over a finite signed alphabet~$\Si$.
}
\end{quote}
  
In addition to the main results of this article, we briefly recall in~\ref{SS:Yamanouchi} and~\ref{SS:MatrixVersion}  two related developments from our earlier works. In~\cite{Hage25}, we established a super analogue of the Littlewood--Richardson rule by introducing \emph{Littlewood--Richardson tableaux} in the signed setting and constructing a bijection that realizes the corresponding structure constants as the number of such tableaux of prescribed shape and weight. In~\cite{Hage26}, we developed a matrix interpretation of signed two-rowed arrays, leading to bijective proofs of \emph{super Cauchy identities} via the super-RSK correspondence. In this framework, parity-dependent restrictions on matrix entries naturally reflect the signed structure of the alphabets, while matrix transposition accounts for the symmetry of the correspondence. Together, these results provide combinatorial models for super Schur functions that extend classical constructions and connect naturally with the theory developed in the present paper.

\subsection*{Notation}
We denote by $\mathcal{A}^\ast$ the free monoid of \emph{words} over a totally ordered alphabet~$\mathcal{A}$, the product being concatenation of words and the identity being the empty word, and by~$\# \mathcal{A}$ the cardinal number of~$\mathcal{A}$ when it is finite. 
We denote by~$[n]$ the ordered set~$\{1<2<\ldots<n\}$ for $n$ in $\Zb_{>0}$.
Let $\Si$ be a finite or countable totally ordered set and~$||.||:\Si\to \mathbb{Z}_2$ be any map, where $\mathbb{Z}_2=\{0,1\}$ denotes the additive cyclic group of order $2$. We call the ordered pair~$(\Si, ||.||)$  a \emph{signed alphabet} and the map~$||.||$ the \emph{signature map} of~$\Si$. We denote~$\Si_0 =\{a\in \Si \;\big|\; ||a||=0\}$ and~$\Si_1 =\{a\in \Si \;\big|\; ||a||=1\}$.  
A monoid~$\M$ is a \emph{$\Zb_2$-graded monoid} or a \emph{supermonoid} if a map~$||.||:\M\to \Zb_2$ is given such that $||u.v||= ||u||+||v||$, for all~$u, v$ in~$\M$. We call $||u||$ the \emph{$\Zb_2$-degree} of~$u$. The free monoid~$\Si^\ast$ over~$\Si$ is  $\Zb_2$-graded  by considering~$||w||:=||x_1||+\ldots+||x_k||$, for any~$w=x_1\ldots x_k$ in~$\Si^\ast$. In the rest of this article, and if there is no possible confusion,~$\Si$ denotes a signed alphabet.

\section{The super plactic monoid of type A}
\label{S:SuperPlacticMonoid}

In this section, we recall the notions of super Young tableaux, the super plactic monoid and the super jeu de taquin from~\cite{LaScalaNardozzaSenato06, Hage2021Super}.
Moreover, we give a super version of the Robinson--Schensted correspondence for super tableaux over a signed alphabet.

\subsection{Super Young tableaux}
\label{SS:SuperYoungTableaux}
Let~$n$ be a positive integer.
A \emph{partition} of~$n$ is a weakly decreasing sequence~$\lambda=(\lambda_1,\ldots,\lambda_k)$ in~$\Zb_{>0}^k$ such that~$\sum\lambda_i = n$. We call the integer~$k$ \emph{number of parts} or \emph{height} of~$\lambda$. We denote by~$\Pr_n$ the set of partitions of~$n$ and we set~$\Pr=\underset{n\in\Zb_{>0}}{\bigcup\Pr_n}$.
The \emph{(Ferrers)--Young diagram} of a partition~$\lambda= (\lambda_1,\ldots, \lambda_k)$ of~$n$,  denoted by~$\Yr(\lambda)$, is the set of pairs~$(i,j)$ such that~$1\leq i\leq k$ and~$1\leq j\leq \lambda_i$, that can be represented by a diagram by drawing a box for each pair~$(i,j)$. For instance, the Young diagram~$\Yr((4,3,1,1))$ is represented by the following diagram:
\[
\ytableausetup{mathmode, boxsize=1em}
\begin{ytableau}
\empty& \empty&\empty&\empty \\
\empty&\empty&\empty\\
\empty&\none\\
\empty&\none
\end{ytableau}.
\]
Let~$\Si$ be a signed alphabet.
For a partition~$\lambda$ in~$\Pr$, its \emph{conjugate partition}~$\widetilde{\lambda}$, is the partition corresponding to the transposed diagram~$\{(j,i)\;\big|\; (i,j)\in\Yr(\lambda)\}$, whose parts are the number of boxes of the columns of~$\Yr(\lambda)$.
A \emph{super semistandard Young tableau}, or \emph{super tableau} for short, of \emph{shape}~$\lambda$ over~$\Si$, is a pair~$t:=(\lambda,\Tc)$, where~$\Tc:\Yr(\lambda)\to\Si$ is a map satisfying the following conditions:
\begin{equation}
\label{Eq:SuperTableaux}
\begin{array}{rl}
\Tc(i,j)\leq \Tc(i,j+1),& \text{ with }~\Tc(i,j)= \Tc(i,j+1) \text{ only if }~||\Tc(i,j)|| = 0,\\
\Tc(i,j)\leq \Tc(i+1,j),& \text{ with }~\Tc(i,j)= \Tc(i+1,j) \text{ only if }~||\Tc(i,j)|| = 1.
\end{array}
\end{equation}
We call~$\Yr(\lambda)$ and~$\Tc$  the \emph{frame} and the \emph{filling} of~$t$, respectively.
We call~$t$ a \emph{standard tableau} over~$\Si$ if~$\Tc$ is injective. We denote by~$\emptyset$ the empty super tableau, by~$\YoungST(\Si)$ the set of all standard tableaux, and by~$\YoungT(\Si)$ (resp.~$\YoungT(\Si,\lambda)$) the set of all super tableaux (resp. of shape~$\lambda$) over~$\Si$.
 
Let~$\lambda$ and~$\mu$ be in~$\Pr$ of heights~$k$ and~$l$, respectively, such that~$l\leq k$. We denote~$\mu\subseteq\lambda$ if~$\mu_i\leq \lambda_i$ for any~$i$, that is, the Young diagram of~$\mu$ is contained in that of~$\lambda$, and we call the following set
\[
\Yr(\lambda/\mu)\;:=\;\big\{(i,j)\;\big|\; 1\leq i\leq k,\; \mu_i<j\leq \lambda_i \big\}
\]
a \emph{skew diagram}~$\lambda/\mu$ or a \emph{skew shape}. We denote~$\lambda/0 :=\lambda$ the skew shape corresponding to the Young diagram~$\Yr(\lambda)$.
Let~$\lambda/\mu$ be a skew shape.
A \emph{super semistandard skew tableau}, or \emph{super skew tableau} for short, of \emph{shape}~$\lambda/\mu$ over~$\Si$, is a pair~$S:=(\lambda/\mu,\U)$, where~$\U: \Yr(\lambda/\mu)\to \Si$ is a map satisfying the conditions~\eqref{Eq:SuperTableaux}. We call~$\Yr(\lambda/\mu)$ and~$\U$ the \emph{frame} and the \emph{filling} of~$S$, respectively.
We denote by~$\SkewT(\Si)$ (resp.~$\SkewT(\Si,\lambda/\mu)$)  the set of all super skew tableaux (resp. of shape~$\lambda/\mu$) over~$\Si$. 
Note that we will identify the set~$\YoungT(\Si,\lambda)$ with the set~$\SkewT(\Si,\lambda/0)$.
Let~$S$ be in~$\SkewT(\Si,\lambda/\mu)$.
If~$S$ is not a super tableau, then it has at least an \emph{inner corner}, that is, a box in~$\Yr(\mu)$ such that the boxes below and to the right are not in~$\Yr(\mu)$. An \emph{outer corner} of~$S$ is a box such that neither box below or to the right is in~$\Yr(\lambda/\mu)$. Note that, in some cases, inner corners of~$S$ are also outer corners of it.
We denote by~$\R:\SkewT(\Si)\to \Si^\ast$ the \emph{column reading map} that reads a super tableau column-wise from bottom to top and from left to right. 

\begin{example}
Consider the alphabet~$\Si= \Zb_{>0}$  with signature given by $\Si_0$  the set of even numbers and~$\Si_1$  defined consequently. 
The following diagram is a super tableau of shape~$(3,2,2,1)$ over~$\Si$:
\begin{equation}
\label{Ex:tableauExample}
\raisebox{-0.45cm}{$t\;
=
\;$}
\ytableausetup{mathmode, boxsize=1em}
\begin{ytableau}
1&2&3\\
1&3\\
2&4\\
4
\end{ytableau}
\qquad
\raisebox{-0.45cm}{$\text{ with }~\R(t)=42114323.$}
\end{equation}
The following diagram is a super skew tableau of shape~$(6,4,2,2,1,1)/(3,2,2,1)$ over~$\Si$ :
\begin{equation}
\label{Ex:SkewtableauExample}
\raisebox{-0.45cm}{$S\;
=
\;$}
\ytableausetup{mathmode, boxsize=1em}
\begin{ytableau}
\empty&\empty&\empty*(BlueD)&1&2&3\\
\empty&\empty&1& *(BlueD) 3\\
\empty&\empty\\
\empty&4\\
2\\
4
\end{ytableau}
\qquad
\raisebox{-0.45cm}{$\text{ with }~\R(S)=42413123,$}
\end{equation}
where the empty blue box and the blue one filled with~$3$ denote respectively an inner corner and an outer corner of~$S$.
\end{example}

\subsection{A super Robinson--Schensted correspondence}
\label{SS:SuperRSCorrespondence}
From a super tableau we can obtain a word over~$\Si$ by taking its column reading. The following algorithm will allow us to invert this process and obtain a super tableau from any word over~$\Si$.

\begin{Algorithm}[\cite{LaScalaNardozzaSenato06}]
\label{A:RightInsertion}
The \emph{right (or row) insertion}, denoted by~$\insr{}$, inserts an element~$x$ in~$\Si$ into a super tableau~$t$ of~$\YoungT(\Si)$ as follows:
~\par\nobreak
\emph{Input:} A super tableau $t$ and a letter $x \in \Si$.

\emph{Output:} A super tableau $t \insr{} x$.

\emph{Method:} If $t$ is empty, create a box and fill it with~$x$. Suppose~$t$ is 
non-empty. If~$x\in\Si_0$ (resp.~$x\in\Si_1$) is at least as large as (resp. larger than) the last element of the top row of $t$, then put~$x$ in a box to the right of this row; 
otherwise let~$y$ be the smallest element of the top row of~$t$ such that~$y> x$ (resp.~$y\geq x$). Then replace~$y$ by~$x$ in this row and recursively insert~$y$ into the super tableau formed by the rows of~$t$ below the topmost. Note that this recursion may end with an insertion into an empty row below the existing rows of $t$.
Output the resulting super tableau.
\end{Algorithm}

For any word $w=x_1\ldots x_k$  over~$\Si$,  a super tableau~$\T(w)$ is computed from~$w$ by inserting its letters iteratively from left to right  starting from the empty super tableau, as follows:
\[
\T(w) 
\;
:=
\;
(\emptyset \insr{} w)
\;
=
\;((\ldots(\emptyset \insr{} x_1)  \insr{} \ldots )\insr{} x_k).
\]
Note that, for any~$t$ in~$\YoungT(\Si)$, the equality~$\T(\R(t)) = t$ holds,~\cite{LaScalaNardozzaSenato06}.

\begin{lemma}[Row-bumping lemma,~\cite{LaScalaNardozzaSenato06}]
\label{L:RowBumpingLemma}
Let~$t$ be in~$\YoungT(\Si)$ and~$x, x'$ be in~$\Si$. Consider the position~$(i,j)$ (resp.~$(i',j')$)   of the new box added to the frame of~$t$ (resp.~$(t\insr{}x)$) after computing~$(t\insr{}x)$ (resp.~$((t\insr{}x)\insr{}x')$).
The following conditions are equivalent:
\begin{enumerate}[\bf i)]
\item $x\leq x'$, with~$x=x'$ only if~$||x|| = 0$,
\item $j<j'$ and~$i\geq i'$.
\end{enumerate}
\end{lemma}

\begin{Algorithm}
\label{A:RSalgo}
Let~$w$ be in~$\Si$.
When computing the super tableau~$\T(w)$,  a standard tableau~$\Q(w)$ over~$[n]$ can be also computed as follows:
~\par\nobreak
\emph{Input:} A word $x_1\ldots x_k$ over $\Si$.

\emph{Output:} A super tableau $\T(x_1\ldots x_k)$ over $\Si$ and a standard tableau~$\Q(x_1\ldots x_k)$ over $[n]$.

\emph{Method:} Start with an empty super tableau~$T_0$ and an empty  standard tableau~$Q_0$. For each $i = 1,\ldots, k$, compute~$T_{i -1}\insr{} x_i$ as per Algorithm~\ref{A:RightInsertion}; let $T_i$ be the resulting super tableau.
Add a box filled with~$i$ to the standard tableau~$Q_{i -1}$ in the same place
as the box that belongs to $T_i$ but not to~$T_{i -1}$; let $Q_i$ be the resulting
standard  tableau.
Output $T_k$ for $\T(x_1\ldots x_k)$ and $Q_k$ as $\Q(x_1\ldots x_k)$.
\end{Algorithm}

\begin{example}
For instance, consider the alphabet~$\Si=\{1,2,3,4,5,6\}$ with signature given by~$\Si_0=\{1,3,4\}$ and~$\Si_1$ defined consequently. The sequence of pairs produced during the computation of $\T(w)$ and~$\Q(w)$ starting from the word~$w=2421456356431541$ is
\[
\big(\emptyset
\;
,
\;
\emptyset\big),
\big(
{\ytableausetup{mathmode, boxsize=1em}
\begin{ytableau}
2
\end{ytableau}}
\;
,
\;
{\ytableausetup{mathmode, boxsize=1em}
\begin{ytableau}
1
\end{ytableau}}
\big), 
\big(
{\ytableausetup{mathmode, boxsize=1em}
\begin{ytableau}
2&4
\end{ytableau}}
\;
,
\;
{\ytableausetup{mathmode, boxsize=1em}
\begin{ytableau}
1&2
\end{ytableau}}
\big), 
\Big(\raisebox{0.1cm}{
{\ytableausetup{mathmode, boxsize=1em}
\begin{ytableau}
2&4\\
2
\end{ytableau}}}
\;
,
\;
\raisebox{0.1cm}{
{\ytableausetup{mathmode, boxsize=1em}
\begin{ytableau}
1&2\\
3
\end{ytableau}
}}\Big), 
\bigg(\raisebox{0.25cm}{
{\ytableausetup{mathmode, boxsize=1em}
\begin{ytableau}
1&4\\
2\\
2
\end{ytableau}}}
\;
,
\;
\raisebox{0.25cm}{
{\ytableausetup{mathmode, boxsize=1em}
\begin{ytableau}
1&2\\
3\\
4
\end{ytableau}}
}\bigg),
\bigg(\raisebox{0.25cm}{
{\ytableausetup{mathmode, boxsize=1em}
\begin{ytableau}
1&4&4\\
2\\
2
\end{ytableau}}}
\;
,
\;
\raisebox{0.25cm}{
{\ytableausetup{mathmode, boxsize=1em}
\begin{ytableau}
1&2&5\\
3\\
4
\end{ytableau}}
}\bigg),
\]
\[
\bigg(\raisebox{0.25cm}{
{\ytableausetup{mathmode, boxsize=1em}
\begin{ytableau}
1&4&4&5\\
2\\
2
\end{ytableau}}}
\;
,
\;
\raisebox{0.25cm}{
{\ytableausetup{mathmode, boxsize=1em}
\begin{ytableau}
1&2&5&6\\
3\\
4
\end{ytableau}}
}\bigg),
\bigg(\raisebox{0.25cm}{
{\ytableausetup{mathmode, boxsize=1em}
\begin{ytableau}
1&4&4&5&6\\
2\\
2
\end{ytableau}}}
\;
,
\;
\raisebox{0.25cm}{
{\ytableausetup{mathmode, boxsize=1em}
\begin{ytableau}
1&2&5&6&7\\
3\\
4
\end{ytableau}}
}\bigg),
\Bigg(\raisebox{0.25cm}{
{\ytableausetup{mathmode, boxsize=1em}
\begin{ytableau}
1&3&4&5&6\\
2&4\\
2
\end{ytableau}}}
\;
,
\;
\raisebox{0.25cm}{
{\ytableausetup{mathmode, boxsize=1em}
\begin{ytableau}
1&2&5&6&7\\
3&8\\
4
\end{ytableau}}
}\Bigg),
\]
\[
\Bigg(\raisebox{0.25cm}{
{\ytableausetup{mathmode, boxsize=1em}
\begin{ytableau}
1&3&4&5&6\\
2&4&5\\
2
\end{ytableau}}}
\;
,
\;
\raisebox{0.25cm}{
{\ytableausetup{mathmode, boxsize=1em}
\begin{ytableau}
1&2&5&6&7\\
3&8&9\\
4
\end{ytableau}}
}\Bigg),
\;
\Bigg(\raisebox{0.25cm}{
{\ytableausetup{mathmode, boxsize=1em}
\begin{ytableau}
1&3&4&5&6\\
2&4&5&6\\
2
\end{ytableau}}}
\;
,
\;
\raisebox{0.25cm}{
{\ytableausetup{mathmode, boxsize=1em}
\begin{ytableau}
1&2&5&6&7\\
3&8&9&10\\
4
\end{ytableau}}
}\Bigg),
\]
\[
\Bigg(\raisebox{0.25cm}{
{\ytableausetup{mathmode, boxsize=1em}
\begin{ytableau}
1&3&4&4&6\\
2&4&5&6\\
2&5
\end{ytableau}}}
\;
,
\;
\raisebox{0.25cm}{
{\ytableausetup{mathmode, boxsize=1em}
\begin{ytableau}
1&2&5&6&7\\
3&8&9&10\\
4&11
\end{ytableau}}
}\Bigg),
\;
\Bigg(\raisebox{0.3cm}{
{\ytableausetup{mathmode, boxsize=1em}
\begin{ytableau}
1&3&3&4&6\\
2&4&4&6\\
2&5\\
5
\end{ytableau}}}
\;
,
\;
\raisebox{0.3cm}{
{\ytableausetup{mathmode, boxsize=1em}
\begin{ytableau}
1&2&5&6&7\\
3&8&9&10\\
4&11\\
12
\end{ytableau}}
}\Bigg),
\]
\[
\Bigg(\raisebox{0.3cm}{
{\ytableausetup{mathmode, boxsize=1em}
\begin{ytableau}
1&1&3&4&6\\
2&3&4&6\\
2&4\\
5\\
5
\end{ytableau}}}
\;
,
\;
\raisebox{0.3cm}{
{\ytableausetup{mathmode, boxsize=1em}
\begin{ytableau}
1&2&5&6&7\\
3&8&9&10\\
4&11\\
12\\
13
\end{ytableau}}
}\Bigg),
\;
\Bigg(\raisebox{0.5cm}{
{\ytableausetup{mathmode, boxsize=1em}
\begin{ytableau}
1&1&3&4&5\\
2&3&4&6\\
2&4&6\\
5\\
5
\end{ytableau}}}
\;
,
\;
\raisebox{0.5cm}{
{\ytableausetup{mathmode, boxsize=1em}
\begin{ytableau}
1&2&5&6&7\\
3&8&9&10\\
4&11&14\\
12\\
13
\end{ytableau}}
}\Bigg),
\]
\[
\Bigg(\raisebox{0.5cm}{
{\ytableausetup{mathmode, boxsize=1em}
\begin{ytableau}
1&1&3&4&4\\
2&3&4&5\\
2&4&6\\
5&6\\
5
\end{ytableau}}}
\;
,
\;
\raisebox{0.5cm}{
{\ytableausetup{mathmode, boxsize=1em}
\begin{ytableau}
1&2&5&6&7\\
3&8&9&10\\
4&11&14\\
12&15\\
13
\end{ytableau}}
}\Bigg),
\;
\Bigg(\raisebox{0.5cm}{
\raisebox{-0.5cm}{$\T(w)\; =\; $}
{\ytableausetup{mathmode, boxsize=1em}
\begin{ytableau}
1&1&1&4&4\\
2&3&3&5\\
2&4&4\\
5&6\\
5&6
\end{ytableau}}}
\;
,
\;
\raisebox{0.5cm}{
\raisebox{-0.5cm}{$\Q(w)\; =\; $}
{\ytableausetup{mathmode, boxsize=1em}
\begin{ytableau}
1&2&5&6&7\\
3&8&9&10\\
4&11&14\\
12&15\\
13&16
\end{ytableau}}
}\Bigg).
\]
\end{example}

\begin{Algorithm}[\cite{LaScalaNardozzaSenato06}]
\label{A:ReverseRightInsertion}
 Let~$t$ be in~$\YoungT(\Si)$ and $x$ in~$\Si$. Starting with the super tableau~$t\insr{} x$, together with the outer corner that has been
added to the frame of~$t$ after insertion, we can recover the original tableau~$t$ and the element~$x$ as follows:
~\par\nobreak
\emph{Input:} A super tableau $t$ and an outer corner of~$t$.

\emph{Output:} A super tableau $t'$ and a letter~$x\in \Si$.

\emph{Method:}  Suppose that $y \in \Si_0$ (resp.~$y \in \Si_1$) is the entry in the outer corner of~$t$. Find in the row above this outer corner the entry farthest to the right which is strictly smaller  (resp. smaller) than~$y$. Then this entry is replaced by~$y$ and it is bumped up to the next row where the process is repeated until an entry is bumped out of
the top row. Output the resulting super tableau for~$t'$ and the last element which is bumped out for~$x$.
\end{Algorithm}

\begin{Algorithm}
\label{A:ReverseRSCorrespodence}
Let~$w$ be in~$\Si$. Starting from the pair of super tableaux~$(\T(w),\Q(w))$
obtained by Algorithm~\ref{A:RSalgo} we can recover the initial word~$w$ as follows:
~\par\nobreak
\emph{Input:} A pair~$(T,Q)\in \YoungT(\Si)\times \YoungST([n])$ of same-shape super tableaux containing $k$ boxes.

\emph{Output:} A word $x_1\ldots x_k$ over $\Si$.

\emph{Method:}
Start with~$T_k= T$ and~$Q_k= Q$. For each~$i=k-1,\ldots, 1$, take the box filled with~$i+1$ in~$Q_{i+1}$, and apply Algorithm~\ref{A:ReverseRightInsertion} to~$T_{i+1}$ with the outer corner corresponding to that box. The resulting tableau is~$T_i$, and the element that is bumped out of the top row of~$T_{i+1}$ is denoted by~$x_{i+1}$. Remove the box containing~$i+1$ from~$Q_{i+1}$ and denote the resulting standard tableau by~$Q_{i}$. Output the word~$x_1 \ldots x_k$.
\end{Algorithm}

This sets up a one-to-one correspondence between words over~$\Si$ and pairs of same-shape tableaux~$(T, Q)$ in~$\YoungT(\Si)\times \YoungST([n])$:

\begin{proposition}
\label{P:RSCorrespondence}
The map~$w \mapsto (\T(w),\Q(w))$ defines a bijection between words over~$\Si$ and pairs consisting of a super tableau over~$\Si$ and a standard tableau over~$[n]$ of the same shape. 
\end{proposition}

This is the super version of the \emph{Robinson--Schensted correspondence} between words over~$[n]$ and pairs consisting of a  semistandard tableau  and a standard tableau over~$[n]$ of the same shape,~\cite{Fulton97}.

\subsection{The super plactic monoid of type A}
\label{SS:SuperplacticMonoid}

The \emph{super plactic monoid} over~$\Si$ is the monoid, denoted by~$\P(\Si)$,  generated by the set~$\Si$ and subject to the following family of \emph{super Knuth relations},~\cite{LaScalaNardozzaSenato06}:
\begin{equation}
\label{Eq:SuperKnuthRelations}
\begin{array}{rl}
xzy=zxy,\;\text{ with }\; x=y \;\text{ only if }\; ||y|| = 0 \;\text{ and } \; y=z\; \text{ only if }\; ||y|| = 1,\\
yxz=yzx,\;\text{ with }\; x=y \;\text{ only if }\;  ||y|| = 1\;\text{ and } \; y=z\; \text{ only if }\; ||y|| =0,
\end{array}
\end{equation}
for any~$x\leq y\leq z$ of elements of~$\Si$. 
The congruence generated by the relations~(\ref{Eq:SuperKnuthRelations}), is denoted by~$\sim_{\P(\Si)}$, and called the \emph{super plactic congruence}.
Two words over~$\Si$ are \emph{super plactic equivalent} if one can be transformed into the other with respect~$\sim_{\P(\Si)}$.
For instance, consider the readings of the super tableau~$t$ presented in~\eqref{Ex:tableauExample} and the super skew tableau~$S$ presented in~\eqref{Ex:SkewtableauExample}, we have
\begin{equation}
\label{Eq:SuperplacticEquivalent}
\R(t) \;=\; 42{\bf 114} 323 \;\sim_{\P(\Si)}\; 4{\bf 214}1323\;\sim_{\P(\Si)}\; 4241{\bf 132}3 \;\sim_{\P(\Si)}\; 42413123 \;=\; \R(S).
\end{equation}
Since the relations~(\ref{Eq:SuperKnuthRelations}) are $\Zb_2$-homogeneous, the monoid~$\P(\Si)$ is a supermonoid. For any~$w$ in~$\Si^\ast$,  we have~$w\sim_{\P(\Si)} \R(\T(w))$,~\cite{ Hage2021Super}. Moreover,  the set~$\YoungT(\Si)$ satisfies the \emph{cross-section property} for~$\sim_{\P(\Si)}$, that is:

\begin{property}[\cite{LaScalaNardozzaSenato06}]
\label{P:CrossSectionProperty}
For all~$w$ and~$w'$ in~$\Si^\ast$, we have~$w \sim_{\P(\Si)} w'$  if and only if~$\T(w) = \T(w')$.
\end{property}

We define an internal product~$\star_{\YoungT(\Si)}$ on $\YoungT(\Si)$ by setting~$t \star_{\YoungT(\Si)} t' :=  (t\insr{} \R(t'))$, for all~$t,t'$ in $\YoungT(\Si)$.
By definition the relations $t\star_{\YoungT(\Si)} \emptyset = t$ and $\emptyset \star_{\YoungT(\Si)} t = t$ hold, showing that the product~$\star_{\YoungT(\Si)}$ is unitary with respect to~$\emptyset$. Following Property~\ref{P:CrossSectionProperty}, we deduce that the product~$\star_{\YoungT(\Si)}$  is associative and  the monoids~$(\YoungT(\Si), \star_{\YoungT(\Si)})$ and~$\P(\Si)$ are isomorphic.

\begin{lemma}
\label{L:SuperplacticLargestSmallest}
Let~$w$ and~$v$ be two words over~$\Si$ and let~$w'$ and~$v'$ be the words obtained respectively from~$w$ and~$v$  after removing the $i$ largest and the $j$ smallest letters for any~$i$ and~$j$. If~$w$ and~$v$ are super plactic equivalent, then~$w'$ and~$v'$ are so.
\end{lemma}

\begin{proof}
It is sufficient by induction to prove that the words obtained by removing the largest
or the smallest letters from~$w$ and $v$ are super plactic equivalent. 
We  prove the case of the largest, the other being similar. 
Suppose that~$w=uxzyu'$ and~$v=uzxyu'$ (resp.~$w=uyxzu'$ and~$v=uyzxu'$) for all~$u, u'$ in~$\Si^\ast$ such that~$x\leq y\leq z$ in~$\Si$ with~$x=y$  only if  $||y|| = 0$ (resp.~$||y|| = 1$)  and~$y=z$ only if~$||y|| = 1$ (resp.~$||y|| = 0$). We  consider two cases depending on whether or not the element that is removed from~$w$ and~$v$ is one of the letters~$x$, $y$, or $z$. If this element is not one
of these letters, then the resulting words are obviously super plactic equivalent. Otherwise, the removed letter must be the
letter $z$ and then the resulting words are equal, showing the claim.
\end{proof}

\subsection{The super jeu de taquin}
\label{SS:SuperJeuDeTaquin}
We recall from~\cite{Hage2021Super} the super jeu de taquin procedure which transforms super skew tableaux into super tableaux over~$\Si$ as follows. 
A \emph{forward sliding} is a sequence of  the following \emph{sliding operations}:
\[
\begin{array}{rl}
\ytableausetup{mathmode, boxsize=1em}
{\begin{ytableau}
*(BlueD)&y\\
   x
\end{ytableau}}
 \;\rightarrow\;
{ \begin{ytableau}
x& y\\
 *(BlueD)&\none
\end{ytableau}}, &\;
\text{ for any } x\leq y \text{ with }  x=y \text{ only if } ||x||=0,
\\
\\
{\begin{ytableau}
 *(BlueD)&x\\
y
\end{ytableau}}
\;\rightarrow\;
{\begin{ytableau}
x& *(BlueD)\\
y&\none 
\end{ytableau}},&\; \text{ for any }~x\leq y \text{ with } x=y \text{ only if } ||x||=1,
\\
\\
{\begin{ytableau}
 *(BlueD)&\empty\\
x
\end{ytableau}}
\;\rightarrow\;
{\begin{ytableau}
x&\empty\\
*(BlueD)&\none 
\end{ytableau}},
&
\quad 
{\begin{ytableau}
*(BlueD)&x\\
\empty&\none
\end{ytableau}}
\;\rightarrow\;
{\begin{ytableau}
x& *(BlueD)\\
\empty&\none 
\end{ytableau}},
\quad
{\begin{ytableau}
*(BlueD) &\empty\\
\empty&\none
\end{ytableau}}
\;\rightarrow\;
{\begin{ytableau}
\empty& *(BlueD)\\
\empty&\none 
\end{ytableau}},\; \text{ for any } x,
\end{array}
\]
 starting from a super skew tableau and one of its inner corners, and moving the empty box until it becomes an outer corner.
The \emph{super jeu de taquin} on a super skew tableau~$S$ consists in applying successively the forward sliding algorithm starting from~$S$ until we obtain a super tableau,  denoted by~$\Rec(S)$, and called the \emph{rectification} of~$S$,~\cite{Hage2021Super}.
Two super  skew  tableaux are \emph{super jeu de taquin equivalent} if one can be obtained from the other by a sequence of sliding operations. 
We show that if two super skew tableaux are super jeu de taquin equivalent, then their reading words are super plactic equivalent,~\cite{Hage2021Super}. Moreover, we show the following property:

\begin{property}[\cite{Hage2021Super}]
\label{P:SuperJeuDeTaquinPlactic}
For all~$S, S'\in \SkewT(\Si)$,~$\Rec(S) = \Rec(S')$ if and only if~$\R(S) \sim_{\P(\Si)} \R(S')$.
\end{property}

We deduce that starting with a given super skew tableau, all choices of inner corners lead to the same rectified super tableau,~\cite{Hage2021Super}. This is the \emph{confluence property} of the super jeu de taquin.

\begin{example}
\label{Ex:SuperJeuDeTaquin}
For instance, the super jeu de taquin on the super skew tableau~$S$ in~(\ref{Ex:SkewtableauExample}) applies eight occurrences of forward sliding:
\[
\raisebox{-0.5cm}{$S\; =\;$}
{\ytableausetup{mathmode, boxsize=1em}
\begin{ytableau}
\empty&\empty&\empty&1&2&3\\
\empty&\empty&1&3\\
\empty&\empty*(RedD)\\
\empty&4\\
2\\
4
\end{ytableau}
\;
\raisebox{-0.5cm}{$\rightarrow$}}
\;
{\ytableausetup{mathmode, boxsize=1em}
\begin{ytableau}
\empty&\empty&\empty&1&2&3\\
\empty&\empty&1&3\\
\empty&4\\
\empty&\empty*(RedD)\\
2\\
4
\end{ytableau}}
\raisebox{-0.5cm}{$;$}
\;
{\ytableausetup{mathmode, boxsize=1em}
\begin{ytableau}
\empty&\empty&\empty&1&2&3\\
\empty&\empty&1&3\\
\empty&4\\
\empty*(BlueD)\\
2\\
4
\end{ytableau}}
\;
\raisebox{-0.5cm}{$\rightarrow$}
\;
{\ytableausetup{mathmode, boxsize=1em}
\begin{ytableau}
\empty&\empty&\empty&1&2&3\\
\empty&\empty&1&3\\
\empty&4\\
2\\
\empty*(BlueD)\\
4
\end{ytableau}}
\;
\raisebox{-0.5cm}{$\rightarrow$}
\;
{\ytableausetup{mathmode, boxsize=1em}
\begin{ytableau}
\empty&\empty&\empty&1&2&3\\
\empty&\empty&1&3\\
\empty&4\\
2\\
4\\
\empty*(BlueD)
\end{ytableau}}
\;
\raisebox{-0.5cm}{$;$}
\]
\[
{\ytableausetup{mathmode, boxsize=1em}
\begin{ytableau}
\empty&\empty&\empty&1&2&3\\
\empty&\empty&1&3\\
\empty*(GreenD)&4\\
2\\
4
\end{ytableau}}
\;
\raisebox{-0.5cm}{$\rightarrow$}
\;
{\ytableausetup{mathmode, boxsize=1em}
\begin{ytableau}
\empty&\empty&\empty&1&2&3\\
\empty&\empty&1&3\\
2&4\\
\empty*(GreenD)\\
4
\end{ytableau}}
\;
\raisebox{-0.5cm}{$\rightarrow$}
\;
{\ytableausetup{mathmode, boxsize=1em}
\begin{ytableau}
\empty&\empty&\empty&1&2&3\\
\empty&\empty&1&3\\
2&4\\
4\\
\empty*(GreenD)
\end{ytableau}}
\raisebox{-0.5cm}{$;$}
\]
\[
{\ytableausetup{mathmode, boxsize=1em}
\begin{ytableau}
\empty&\empty&\empty&1&2&3\\
\empty&\empty*(Blue)&1&3\\
2&4\\
4
\end{ytableau}}
\;
\raisebox{-0.5cm}{$\rightarrow$}
\;
{\ytableausetup{mathmode, boxsize=1em}
\begin{ytableau}
\empty&\empty&\empty&1&2&3\\
\empty&1&\empty*(Blue)&3\\
2&4\\
4
\end{ytableau}}
\;
\raisebox{-0.5cm}{$\rightarrow$}
\;
{\ytableausetup{mathmode, boxsize=1em}
\begin{ytableau}
\empty&\empty&\empty&1&2&3\\
\empty&1&3&\empty*(Blue)\\
2&4\\
4
\end{ytableau}}
\raisebox{-0.5cm}{$;$}
\]
\[
{\ytableausetup{mathmode, boxsize=1em}
\begin{ytableau}
\empty&\empty&\empty*(GreenL)&1&2&3\\
\empty&1&3\\
2&4\\
4
\end{ytableau}}
\;
\raisebox{-0.5cm}{$\rightarrow$}
\;
{\ytableausetup{mathmode, boxsize=1em}
\begin{ytableau}
\empty&\empty&1&\empty*(GreenL)&2&3\\
\empty&1&3\\
2&4\\
4
\end{ytableau}}
\;
\raisebox{-0.5cm}{$\rightarrow$}
\;
{\ytableausetup{mathmode, boxsize=1em}
\begin{ytableau}
\empty&\empty&1&2&\empty*(GreenL)&3\\
\empty&1&3\\
2&4\\
4
\end{ytableau}}
\;
\raisebox{-0.5cm}{$\rightarrow$}
\;
{\ytableausetup{mathmode, boxsize=1em}
\begin{ytableau}
\empty&\empty&1&2&3&\empty*(GreenL)\\
\empty&1&3\\
2&4\\
4
\end{ytableau}}
\raisebox{-0.5cm}{$;$}
\]
\[
{\ytableausetup{mathmode, boxsize=1em}
\begin{ytableau}
\empty&\empty*(YellowD)&1&2&3\\
\empty&1&3\\
2&4\\
4
\end{ytableau}}
\;
\raisebox{-0.5cm}{$\rightarrow$}
\;
{\ytableausetup{mathmode, boxsize=1em}
\begin{ytableau}
\empty&1&\empty*(YellowD)&2&3\\
\empty&1&3\\
2&4\\
4
\end{ytableau}}
\;
\raisebox{-0.5cm}{$\rightarrow$}
\;
{\ytableausetup{mathmode, boxsize=1em}
\begin{ytableau}
\empty&1&2&\empty*(YellowD)&3\\
\empty&1&3\\
2&4\\
4
\end{ytableau}}
\;
\raisebox{-0.5cm}{$\rightarrow$}
\;
{\ytableausetup{mathmode, boxsize=1em}
\begin{ytableau}
\empty&1&2&3&\empty*(YellowD)\\
\empty&1&3\\
2&4\\
4
\end{ytableau}}
\raisebox{-0.5cm}{$;$}
\]
\[
{\ytableausetup{mathmode, boxsize=1em}
\begin{ytableau}
\empty&1&2&3\\
\empty*(red)&1&3\\
2&4\\
4
\end{ytableau}}
\;
\raisebox{-0.5cm}{$\rightarrow$}
\;
{\ytableausetup{mathmode, boxsize=1em}
\begin{ytableau}
\empty&1&2&3\\
1&\empty*(red)&3\\
2&4\\
4
\end{ytableau}}
\;
\raisebox{-0.5cm}{$\rightarrow$}
\;
{\ytableausetup{mathmode, boxsize=1em}
\begin{ytableau}
\empty&1&2&3\\
1&3&\empty*(red)\\
2&4\\
4
\end{ytableau}}
\;
\raisebox{-0.5cm}{$;$}
\]
\[
{\ytableausetup{mathmode, boxsize=1em}
\begin{ytableau}
\empty *(cyan)&1&2&3\\
1&3\\
2&4\\
4
\end{ytableau}}
\;
\raisebox{-0.5cm}{$\rightarrow$}
\;
{\ytableausetup{mathmode, boxsize=1em}
\begin{ytableau}
1&\empty *(cyan)&2&3\\
1&3\\
2&4\\
4
\end{ytableau}}
\;
\raisebox{-0.5cm}{$\rightarrow$}
\;
{\ytableausetup{mathmode, boxsize=1em}
\begin{ytableau}
1&2&\empty *(cyan)&3\\
1&3\\
2&4\\
4
\end{ytableau}}
\;
\raisebox{-0.5cm}{$\rightarrow$}
\;
{\ytableausetup{mathmode, boxsize=1em}
\begin{ytableau}
1&2&3&\empty *(cyan)\\
1&3\\
2&4\\
4
\end{ytableau}}
\raisebox{-0.5cm}{$\;=\Rec(S)$}
\]
with $\R(S)=42413123  \sim_{\P(\Si)} 42114323=  \R(\Rec(S))$, as shown in~\eqref{Eq:SuperplacticEquivalent}.
\end{example}

\subsubsection{The reverse sliding}
\label{SSS:ReverseSlidings}
A \emph{reverse sliding} is a sequence of the following \emph{reverse sliding operations},~\cite{Hage2021Super}:
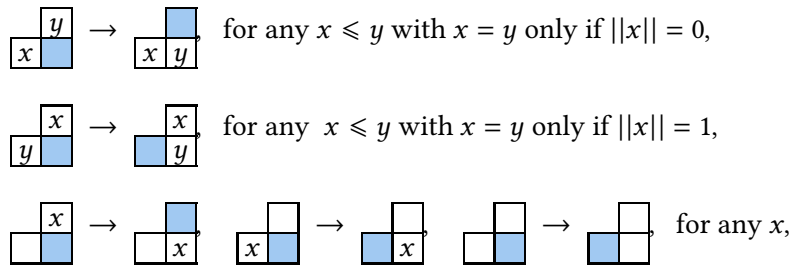
\[
\begin{array}{rl}
\ytableausetup{mathmode, boxsize=1em}
\begin{ytableau}
 \none &y\\
   x&*(BlueD)
\end{ytableau}
 \;\rightarrow\;
 \begin{ytableau}
\none &*(BlueD)\\
   x&y
\end{ytableau},
&\;
\text{ for any }  x\leq y \text{ with } x=y \text{ only if } ||x||=0,
\\
\\
\begin{ytableau}
\none &x\\
y&*(BlueD)
\end{ytableau}
\;\rightarrow\;
{\begin{ytableau}
\none &x\\
*(BlueD)&y
\end{ytableau}},
&\; \text{ for any }~x\leq y \text{ with }  x=y \text{ only if } ||x||=1,
\\
\\
\begin{ytableau}
 \none &x\\
   \empty&*(BlueD)
\end{ytableau}
 \;\rightarrow\;
 \begin{ytableau}
\none &*(BlueD)\\
 \empty&x
\end{ytableau},
&
\quad
\begin{ytableau}
 \none &\empty\\
   x&*(BlueD)
\end{ytableau}
 \;\rightarrow\;
 \begin{ytableau}
 \none &\empty\\
 *(BlueD) & x
\end{ytableau},
\quad
\begin{ytableau}
 \none &\empty\\
  \empty&*(BlueD)
\end{ytableau}
 \;\rightarrow\;
 \begin{ytableau}
 \none &\empty\\
 *(BlueD)& \empty
\end{ytableau}, \; \text{ for any } x,
\end{array}
\] 
starting from a super skew tableau and one of its empty outer corners, and moving this box until it becomes an inner corner.
Starting from the resulting super skew tableau of a forward sliding, together with
the outer corner that was removed, and by applying the reverse sliding, we recover
the initial super skew tableau with the chosen inner corner.

\subsubsection{Super jeu de taquin and insertion}
\label{SSS:InsertionsSuperJeuDeTaquin}
Let $S$ and~$S'$ be super skew tableaux in~$\SkewT(\Si)$ of shape $(\lambda_1,\ldots,\lambda_k)/(\lambda'_1,\ldots,\lambda'_{k'})$ and~$(\mu_1,\ldots,\mu_l)/(\mu'_1,\ldots,\mu'_{l'})$, respectively. We denote by~$[S,S']$ the super skew tableau 
of shape
\[
(\mu_1+\lambda_1,\ldots,\mu_l+\lambda_1, \lambda_1,\ldots,\lambda_{k}) /(\mu'_1+\lambda_1,\ldots, \mu'_{l'}+\lambda_1, \lambda_1,\ldots, \lambda_1, \lambda'_1,\ldots, \lambda'_{k'}),
\]
obtained by concatenating $S'$ over~$S$, as illustrated in the following diagram:
\[
\raisebox{1.45cm}{$[S,S']\;=\;$}
\scalebox{0.6}{
\begin{tikzpicture}
\draw  (0,0) --++(0,1)  --++(0.5,0)-- ++(0,1)-- ++(2,0)-- ++(0,-0.75)-- ++(-1,0)-- ++(0,-1.25)-- ++(-1.5,0)-- ++(-2,0)-- ++(0,-1)-- ++(-0.5,0)-- ++(0,-1)-- ++(1.5,0)-- ++(0,1.25)-- ++(1,0)-- ++(0,0.75);
\draw (0.7,1) node[right]{\scalebox{1.5} {$S^\prime$}} ;
\draw (-1.75,-1) node[right]{\scalebox{1.5} {$S$}};
\end{tikzpicture}}
\raisebox{1.5cm}{$.$}
\]
We define the \emph{insertion product}~$\star_{\SkewT(\Si)}: \SkewT(\Si)\times \SkewT(\Si)\to \YoungT(\Si)$  by setting $S \star_{\SkewT(\Si)} S' :=  (\emptyset\insr{} \R(S) \R(S'))$, for all $S,S'$ in $\SkewT(\Si)$.
This insertion product satisfies the following equality~$S\star_{\SkewT(\Si)} S' = \Rec([S,S'])$
in~$\YoungT(\Si)$, for all~$S, S'$ in~$\SkewT(\Si)$,~\cite{Hage2021Super}.

\section{A super-RSK correspondence with symmetry}
\label{S:SuperRSKCorrespondence}
 In this section, we introduce a super version of the RSK correspondence over a signed version. Using a matrix-ball interpretation, we show that  this correspondence satisfies the symmetry property in complete generality. 

\subsection{A super-RSK correspondence}
\label{SS:SuperRSKCorrespondence}

Let~$(\Si,||.||_1)$ and~$(\Si', ||.||_2)$ be signed alphabets. We define an alphabet structure on the product set~$\Si\times \Si'$ by considering the following  order~$<$ defined by
\begin{equation}
\label{E:Order1}
(x_1,y_1)< (x_2,y_2)\; \text{ if } \; x_1<x_2 \; \text{ or } \; \left\{
    \begin{array}{ll}
        x_1=x_2 \in \Si_0& \mbox{\; and\; } y_1<y_2 \\
        x_1=x_2\in \Si_1& \mbox{\; and\; } y_2<y_1.
    \end{array}
\right.
\end{equation}
Moreover, the signature map~$||.||: \Si\times \Si' \to \Zb_2$ is defined by~$||(x,y)||= ||x||_1+ ||y||_2$. 
From now on, and if there is no possible confusion, we will denote by~$||.||$  the signature maps of~$\Si$,~$\Si'$ and~$\Si\times \Si'$.
A \emph{signed two-rowed array on the alphabets}~$\Si$ and~$\Si'$ is a $2\times k$ matrix 
\[
\begin{pmatrix}
 x_1&\ldots &x_k\\
 y_1&\ldots &y_k
\end{pmatrix}
\]
with~$x_i\in \Si$ and~$y_i\in\Si'$, for all~$i=1,\ldots, k$, such that the following condition holds:
\begin{equation}
\label{E:Order2}
(x_i, y_i) \leq (x_{i+1}, y_{i+1}),  \; \text{ and } \; (x_i, y_i)= (x_{i+1}, y_{i+1}) \; \text{ only if } \; ||(x_i, y_i)|| = 0.
\end{equation}

\begin{Algorithm}
\label{A:RSkalgo}
Let~$\Si$ and~$\Si'$ be  signed alphabets. Starting from a signed two-rowed array~$w$ on~$\Si$ and~$\Si'$, we compute a pair of same-shape super tableaux~$(\T(w), \Q(w))$  whose entries are the
ones of the second and the first row of~$w$, respectively, as follows:
~\par\nobreak
\emph{Input:} A signed two-rowed array~$w= \begin{pmatrix}
 x_1&\ldots &x_k\\
 y_1&\ldots &y_k
\end{pmatrix}$
on $\Si$ and~$\Si'$.

\emph{Output:} A pair~$(\T(w), \Q(w))\in \YoungT(\Si')\times \YoungT(\Si)$ of same-shape super tableaux with~$k$ boxes.

\emph{Method:} Start with an empty super tableau~$T_0$ and an empty  super tableau~$Q_0$. For each $i = 1,\ldots, k$, compute~$T_{i -1}\insr{} y_i$ as per Algorithm~\ref{A:RightInsertion} and let $T_i$ be the resulting super tableau.
Add a box filled with~$x_i$ to the super  tableau~$Q_{i -1}$ in the same place
as the box that belongs to $T_i$ but not to~$T_{i -1}$; let $Q_i$ be the resulting
super  tableau.
Output $T_k$ for $\T(w)$ and $Q_k$ as $\Q(w)$.
\end{Algorithm}

\begin{proposition}
\label{P:RSKalgoCorrectness}
Algorithm~\ref{A:RSkalgo} always halts with the correct output.
\end{proposition}

\begin{proof}
By construction, we have that~$\T(w)$ is a super tableau over~$\Si'$ since it is equal to the super tableau~$\T(y_1\ldots y_k)$. 
We still have to show that~$\Q(w)$ is a super tableau over~$\Si$.
It is sufficient to show by induction that each~$\Q_i$, for~$i = 1,\ldots, k$,  is a super tableau over~$\Si$. By definition of the signed two-rowed array, if~$x_i\in \Si_0$  is placed to the right of (resp. under or to the right of) an entry~$x_l\in\Si_0$ (resp.~$x_l\in\Si_1$) in~$Q_{i-1}$, then~$x_i$ is larger  (resp. strictly larger) than~$x_l$. Similarly, if~$x_i\in \Si_1$  is placed under (resp. under or to the right of) an entry~$x_l\in\Si_1$ (resp.~$x_l\in\Si_0$) in~$Q_{i-1}$, then~$x_i$ is larger  (resp. strictly larger) than~$x_l$. Let us now show  that if~$x_i\in \Si_0$  is placed under  of an entry~$x_l\in\Si_0$ in~$Q_{i-1}$, then~$x_i$ is strictly larger than~$x_l$. Suppose the contrary. Then, we have~$x_i=x_l$ and thus by Conditions~\eqref{E:Order1} and~\eqref{E:Order2} we have:
\[
y_l\leq y_{l+1}\leq\ldots \leq y_i,
\]
with~$y_l = y_{l+1}=\ldots = y_i$ only if~$||y_l|| = 0$. Hence, following Algorithm~\ref{A:RightInsertion}, all the boxes added starting from~$T_l$ to~$T_i$ must be in different columns which contradicts the fact that~$x_i$ and~$x_l$ belong to the same column. Similarly, we show that if~$x_i\in \Si_1$  is placed to the right  of an entry~$x_l\in\Si_1$ in~$Q_{i-1}$, then~$x_i$ is strictly larger than~$x_l$. Note finally that, by construction, the super tableaux~$\T(w)$ and~$\Q(w)$ are of the same shape and contain~$k$ boxes.
\end{proof}

\begin{Algorithm}
\label{A:ReverseRSKCorrespodence}
Let~$w$ be a signed two-rowed array on signed alphabets~$\Si$ and~$\Si'$. Starting from the pair~$(\T(w),\Q(w))$
constructed by Algorithm~\ref{A:RSkalgo}, we can recover the initial signed two-rowed array~$w$ as follows:
~\par\nobreak
\emph{Input:} A pair~$(T,Q)\in \YoungT(\Si')\times \YoungT(\Si)$ of same-shape super tableaux containing~$k$ boxes.

\emph{Output:} A signed two-rowed array~$w= \begin{pmatrix}
 x_1&\ldots &x_k\\
 y_1&\ldots &y_k
\end{pmatrix}$
on $\Si$ and~$\Si'$.

\emph{Method:}
Start with~$T_k= T$ and~$Q_k= Q$. For each~$i=k-1,\ldots, 1$, take the box filled with the largest entry~$x_{i+1}$ in~$Q_{i+1}$ and if there are several equal entries in~$\Si_0$ (resp.~$\Si_1$), the box that is farthest to the right (resp. left) is selected; apply Algorithm~\ref{A:ReverseRightInsertion} to~$T_{i+1}$ with the outer corner corresponding to that box. Let~$T_i$ be the resulting super tableau and~$y_{i+1}$ be the element that is bumped out of the top row of~$T_{i+1}$. Remove the box containing~$x_{i+1}$ from~$Q_{i+1}$ and denote the resulting super tableau tableau by~$Q_{i}$. Output the~$2\times k$ matrix~$\begin{pmatrix}
 x_1&\ldots &x_k\\
 y_1&\ldots &y_k
\end{pmatrix}$.
\end{Algorithm}

\begin{proposition}
\label{P:ReverseRSKalgoCorrectness}
Algorithm~\ref{A:ReverseRSKCorrespodence} always halts with the correct output.
\end{proposition}

\begin{proof}
By construction, we have~$x_1\leq \ldots \leq x_k$.
Suppose now that~$x_{i-1} = x_i$, for~$i\in [k]$ with~$||x_i|| = 0$. Then by constructing, the box that is removed from~$T_i$ lies strictly to the right of the box that is removed from~$T_{i-1}$ in the next step.
Following Lemma~\ref{L:RowBumpingLemma}, the entry~$y_i$ removed first is at least as large as the entry~$y_{i-1}$ removed second, with~$y_i = y_{i-1}$, only if~$||y_i|| = 0$.
Suppose now that~$x_{i-1} = x_i$, for~$i\in [k]$ with~$||x_i|| = 1$. Then by constructing, the box that is removed from~$T_i$ lies strictly below the box that is removed from~$T_{i-1}$ in the next step.
Following Lemma~\ref{L:RowBumpingLemma}, we obtain that~$y_{i-1}\geq y_i$, with~$y_i = y_{i-1}$, only if~$||y_i|| = 1$. Hence,  the output~$2\times k$ matrix satisfies Conditions~\eqref{E:Order1} and~\eqref{E:Order2}, and then it is a signed two-rowed array on~$\Si$ and~$\Si'$.
\end{proof}

It is clear from the constructions that the two processes  described by Algorithms~\ref{A:RSkalgo}  and~\ref{A:ReverseRSKCorrespodence} are inverse to each other.
This sets up a one-to-one correspondence between signed two-rowed arrays  and pairs of same-shape super tableaux on signed alphabets~$\Si$ and~$\Si'$. We will denote by
\[
\RSK: w \mapsto (\T(w),\Q(w))
\]
this mapping. Hence, we obtain the following first main result of the article:

\begin{theorem}[Super-RSK correspondence]
\label{T:SuperRSKCorrespondence}
Let $\Si$ and~$\Si'$ be  signed alphabets. The map~$\RSK$ defines a  one-to-one correspondence  between signed two-rowed arrays and pairs of super
 tableaux on~$\Si$ and~$\Si'$, such that for any signed two-rowed array~$w$ on~$\Si$ and~$\Si'$, we have that $\T(w)$ and~$\Q(w)$ are same-shape super tableaux whose entries are the
ones of the second and the first row of~$w$, respectively.
\end{theorem}

\begin{example}
\label{Ex:RSKExample1}
Consider~$\Si= \Si'=[6]$ with signature given by~$\Si_{0}=\Si'_{0}= \{1,3,4\}$ and~$\Si_1=\Si'_1$ defined consequently.
The sequence of pairs produced during the computation of~$(\T(w), \Q(w))$ starting from the signed two-rowed array~$w=\setcounter{MaxMatrixCols}{20}
\begin{pmatrix}
 1&2&2&2&3&3&3&4&4&4&5&5&5&6&6&6\\
 2&4&2&1&4&5&6&3&5&6&4&3&1&5&4&1
\end{pmatrix}$
on~$\Si$ and~$\Si'$ is the following:
\[
\big(\emptyset
\;
,
\;
\emptyset\big),
\big(
{\ytableausetup{mathmode, boxsize=1em}
\begin{ytableau}
2
\end{ytableau}}
\;
,
\;
{\ytableausetup{mathmode, boxsize=1em}
\begin{ytableau}
1
\end{ytableau}}
\big),
\big(
{\ytableausetup{mathmode, boxsize=1em}
\begin{ytableau}
2&4
\end{ytableau}}
\;
,
\;
{\ytableausetup{mathmode, boxsize=1em}
\begin{ytableau}
1&2
\end{ytableau}}
\big),
\Big(\raisebox{0.1cm}{
{\ytableausetup{mathmode, boxsize=1em}
\begin{ytableau}
2&4\\
2
\end{ytableau}}}
\;
,
\;
\raisebox{0.1cm}{
{\ytableausetup{mathmode, boxsize=1em}
\begin{ytableau}
1&2\\
2
\end{ytableau}
}}\Big),
\bigg(\raisebox{0.25cm}{
{\ytableausetup{mathmode, boxsize=1em}
\begin{ytableau}
1&4\\
2\\
2
\end{ytableau}}}
\;
,
\;
\raisebox{0.25cm}{
{\ytableausetup{mathmode, boxsize=1em}
\begin{ytableau}
1&2\\
2\\
2
\end{ytableau}}
}\bigg),
\bigg(\raisebox{0.25cm}{
{\ytableausetup{mathmode, boxsize=1em}
\begin{ytableau}
1&4&4\\
2\\
2
\end{ytableau}}}
\;
,
\;
\raisebox{0.25cm}{
{\ytableausetup{mathmode, boxsize=1em}
\begin{ytableau}
1&2&3\\
2\\
2
\end{ytableau}}
}\bigg),
\]
\[
\bigg(\raisebox{0.25cm}{
{\ytableausetup{mathmode, boxsize=1em}
\begin{ytableau}
1&4&4&5\\
2\\
2
\end{ytableau}}}
\;
,
\;
\raisebox{0.25cm}{
{\ytableausetup{mathmode, boxsize=1em}
\begin{ytableau}
1&2&3&3\\
2\\
2
\end{ytableau}}
}\bigg),
\bigg(\raisebox{0.25cm}{
{\ytableausetup{mathmode, boxsize=1em}
\begin{ytableau}
1&4&4&5&6\\
2\\
2
\end{ytableau}}}
\;
,
\;
\raisebox{0.25cm}{
{\ytableausetup{mathmode, boxsize=1em}
\begin{ytableau}
1&2&3&3&3\\
2\\
2
\end{ytableau}}
}\bigg),
\Bigg(\raisebox{0.25cm}{
{\ytableausetup{mathmode, boxsize=1em}
\begin{ytableau}
1&3&4&5&6\\
2&4\\
2
\end{ytableau}}}
\;
,
\;
\raisebox{0.25cm}{
{\ytableausetup{mathmode, boxsize=1em}
\begin{ytableau}
1&2&3&3&3\\
2&4\\
2
\end{ytableau}}
}\Bigg),
\]
\[
\Bigg(\raisebox{0.25cm}{
{\ytableausetup{mathmode, boxsize=1em}
\begin{ytableau}
1&3&4&5&6\\
2&4&5\\
2
\end{ytableau}}}
\;
,
\;
\raisebox{0.25cm}{
{\ytableausetup{mathmode, boxsize=1em}
\begin{ytableau}
1&2&3&3&3\\
2&4&4\\
2
\end{ytableau}}
}\Bigg),
\;
\Bigg(\raisebox{0.25cm}{
{\ytableausetup{mathmode, boxsize=1em}
\begin{ytableau}
1&3&4&5&6\\
2&4&5&6\\
2
\end{ytableau}}}
\;
,
\;
\raisebox{0.25cm}{
{\ytableausetup{mathmode, boxsize=1em}
\begin{ytableau}
1&2&3&3&3\\
2&4&4&4\\
2
\end{ytableau}}
}\Bigg),
\]
\[
\Bigg(\raisebox{0.25cm}{
{\ytableausetup{mathmode, boxsize=1em}
\begin{ytableau}
1&3&4&4&6\\
2&4&5&6\\
2&5
\end{ytableau}}}
\;
,
\;
\raisebox{0.25cm}{
{\ytableausetup{mathmode, boxsize=1em}
\begin{ytableau}
1&2&3&3&3\\
2&4&4&4\\
2&5
\end{ytableau}}
}\Bigg),
\;
\Bigg(\raisebox{0.3cm}{
{\ytableausetup{mathmode, boxsize=1em}
\begin{ytableau}
1&3&3&4&6\\
2&4&4&6\\
2&5\\
5
\end{ytableau}}}
\;
,
\;
\raisebox{0.3cm}{
{\ytableausetup{mathmode, boxsize=1em}
\begin{ytableau}
1&2&3&3&3\\
2&4&4&4\\
2&5\\
5
\end{ytableau}}
}\Bigg),
\]
\[
\Bigg(\raisebox{0.3cm}{
{\ytableausetup{mathmode, boxsize=1em}
\begin{ytableau}
1&1&3&4&6\\
2&3&4&6\\
2&4\\
5\\
5
\end{ytableau}}}
\;
,
\;
\raisebox{0.3cm}{
{\ytableausetup{mathmode, boxsize=1em}
\begin{ytableau}
1&2&3&3&3\\
2&4&4&4\\
2&5\\
5\\
5
\end{ytableau}}
}\Bigg),
\;
\Bigg(\raisebox{0.5cm}{
{\ytableausetup{mathmode, boxsize=1em}
\begin{ytableau}
1&1&3&4&5\\
2&3&4&6\\
2&4&6\\
5\\
5
\end{ytableau}}}
\;
,
\;
\raisebox{0.5cm}{
{\ytableausetup{mathmode, boxsize=1em}
\begin{ytableau}
1&2&3&3&3\\
2&4&4&4\\
2&5&6\\
5\\
5
\end{ytableau}}
}\Bigg),
\]
\[
\Bigg(\raisebox{0.5cm}{
{\ytableausetup{mathmode, boxsize=1em}
\begin{ytableau}
1&1&3&4&4\\
2&3&4&5\\
2&4&6\\
5&6\\
5
\end{ytableau}}}
\;
,
\;
\raisebox{0.5cm}{
{\ytableausetup{mathmode, boxsize=1em}
\begin{ytableau}
1&2&3&3&3\\
2&4&4&4\\
2&5&6\\
5&6\\
5
\end{ytableau}}
}\Bigg),
\;
\Bigg(\raisebox{0.5cm}{
\raisebox{-0.5cm}{$\T(w)\; =\; $}
{\ytableausetup{mathmode, boxsize=1em}
\begin{ytableau}
1&1&1&4&4\\
2&3&3&5\\
2&4&4\\
5&6\\
5&6
\end{ytableau}}}
\;
,
\;
\raisebox{0.5cm}{
\raisebox{-0.5cm}{$\Q(w)\; =\; $}
{\ytableausetup{mathmode, boxsize=1em}
\begin{ytableau}
1&2&3&3&3\\
2&4&4&4\\
2&5&6\\
5&6\\
5&6
\end{ytableau}}
}\Bigg).
\]
By applying Algorithm~\ref{A:ReverseRSKCorrespodence} on~$(\T(w),\Q(w))$,  we recover the initial signed two-rowed array~$w$.
\end{example}

\begin{example}
\label{Ex:RSKExample2}
Let~$\Si=\Si' = \Zb_{>0}$  with signature given by $\Si_0$ (resp~$\Si'_1$) the set of even numbers and~$\Si_1$ (resp.~$\Si'_0$) defined consequently. The sequence of pairs produced during the computation of~$(\T(w)$ and~$\Q(w)$ starting from the signed two-rowed array~$w=\setcounter{MaxMatrixCols}{20}
\begin{pmatrix}
 1&1&2&2&3&3&4&4\\
 3&2&1&2&4&3&1&2
\end{pmatrix}$ on~$\Si$ and~$\Si'$ is the following:
\[
\big(\emptyset
\;
,
\;
\emptyset\big),
\;
\big(
{\ytableausetup{mathmode, boxsize=1em}
\begin{ytableau}
3
\end{ytableau}}
\;
,
\;
{\ytableausetup{mathmode, boxsize=1em}
\begin{ytableau}
1
\end{ytableau}}
\big),
\;
\big(
{\ytableausetup{mathmode, boxsize=1em}
\begin{ytableau}
2\\ 
3
\end{ytableau}}
\;
,
\;
{\ytableausetup{mathmode, boxsize=1em}
\begin{ytableau}
1\\
1
\end{ytableau}}
\big),
\;
\Big(\raisebox{0.1cm}{
{\ytableausetup{mathmode, boxsize=1em}
\begin{ytableau}
1\\
2\\
3
\end{ytableau}}}
\;
,
\;
\raisebox{0.1cm}{
{\ytableausetup{mathmode, boxsize=1em}
\begin{ytableau}
1\\
1\\
2
\end{ytableau}
}}\Big),
\;
\bigg(\raisebox{0.25cm}{
{\ytableausetup{mathmode, boxsize=1em}
\begin{ytableau}
1&2\\
2\\
3
\end{ytableau}}}
\;
,
\;
\raisebox{0.25cm}{
{\ytableausetup{mathmode, boxsize=1em}
\begin{ytableau}
1&2\\
1\\
2
\end{ytableau}}
}\bigg),
\;
\bigg(\raisebox{0.25cm}{
{\ytableausetup{mathmode, boxsize=1em}
\begin{ytableau}
1&2&4\\
2\\
3
\end{ytableau}}}
\;
,
\;
\raisebox{0.25cm}{
{\ytableausetup{mathmode, boxsize=1em}
\begin{ytableau}
1&2&3\\
1\\
2
\end{ytableau}}
}\bigg),
\]
\[
\bigg(\raisebox{0.25cm}{
{\ytableausetup{mathmode, boxsize=1em}
\begin{ytableau}
1&2&3\\
2&4\\
3
\end{ytableau}}}
\;
,
\;
\raisebox{0.25cm}{
{\ytableausetup{mathmode, boxsize=1em}
\begin{ytableau}
1&2&3\\
1&3\\
2
\end{ytableau}}
}\bigg),
\;
\bigg(\raisebox{0.25cm}{
{\ytableausetup{mathmode, boxsize=1em}
\begin{ytableau}
1&1&3\\
2&4\\
2\\
3
\end{ytableau}}}
\;
,
\;
\raisebox{0.25cm}{
{\ytableausetup{mathmode, boxsize=1em}
\begin{ytableau}
1&2&3\\
1&3\\
2\\
4
\end{ytableau}}
}\bigg),
\;
\bigg(
\raisebox{0.25cm}{
\raisebox{-0.2cm}{$\T(w)\; =\; $}
{\ytableausetup{mathmode, boxsize=1em}
\begin{ytableau}
1&1&2\\
2&3\\
2&4\\
3
\end{ytableau}}}
\;
,
\;
\raisebox{0.25cm}{
\raisebox{-0.2cm}{$\Q(w)\; =\; $}
{\ytableausetup{mathmode, boxsize=1em}
\begin{ytableau}
1&2&3\\
1&3\\
2&4\\
4
\end{ytableau}}
}\bigg).
\]
By applying Algorithm~\ref{A:ReverseRSKCorrespodence} on~$(\T(w),\Q(w))$,  we recover the initial signed two-rowed array~$w$. 
\end{example}

\subsubsection{The symmetry property}
Let~$w$ be a signed two-rowed array on signed alphabets~$\Si$ and~$\Si'$.
The \emph{inverse} of~$w$, denoted by~$w^{\text{inv}}$, is the  signed two-rowed array on~$\Si'$ and~$\Si$
obtained from~$w$ by \emph{exchanging} the rows of~$w$, that is, writing the second row of~$w$ as the first row and the first row of $w$ as the second row, and by sorting the new couples on~$\Si'\times\Si$ according to Conditions~\eqref{E:Order1} and~\eqref{E:Order2}. 
For instance, the inverse of the signed two-rowed array~$w$ of Example~\ref{Ex:RSKExample1} (resp.~Example~\ref{Ex:RSKExample2}) is the following:
 \[w^{\text{inv}}\;=\;
\setcounter{MaxMatrixCols}{20}
\begin{pmatrix}
 1&1&1&2&2&3&3&4&4&4&4&5&5&5&6&6\\
 2&5&6&2&1&4&5&2&3&5&6&6&4&3&4&3
\end{pmatrix}
\quad
\big(\text{ resp. } w^{\text{inv}}\;=\;
\setcounter{MaxMatrixCols}{20}
\begin{pmatrix}
 1&1&2&2&2&3&3&4\\
 2&4&4&2&1&1&3&3
\end{pmatrix}\big).
\]

We say that~$w$ has \emph{symmetry with respect to the map}~$\RSK$  if it satisfies the following property:

\begin{property} 
If~$\RSK(w)= (\T(w),\Q(w))$  then~$\RSK(w^{\text{inv}}) = (\Q(w),\T(w))$.
\end{property}

\begin{theorem}[Symmetry of the super-RSK correspondence]
\label{T:SymmetryRSK}
Let~$\Si$ and~$\Si'$ be signed alphabets. All signed two-rowed arrays on~$\Si$ and~$\Si'$ have symmetry with respect to the super-RSK correspondence map~$\RSK$.
\end{theorem}

The rest of this section is devoted to prove this result. We  give a geometrical interpretation of the super-RSK correspondence following Fulton's matrix-ball construction,~\cite{Fulton97}, for the non-signed case.
This construction will allow us to prove Theorem~\ref{T:SymmetryRSK} using the symmetry of the resulting signed ball array corresponding to a signed two-rowed array.

\subsection{Super matrix-ball construction}
\label{SS:SuperMatrixBallConstruction}
Let~$\Si$ and~$\Si'$ be signed alphabets.
A \emph{signed ball array on~$\Si$ and~$\Si'$} is a rectangular array of balls filled with positive integers, whose rows (resp. columns) are indexed with elements of~$\Si$ (resp.~$\Si'$), from the smaller to the bigger one,  allowing the repetition only of elements in~$\Si_1$ (resp.~$\Si'_1$) and where many balls can occur in the same position. 
A signed ball array  is  \emph{empty} if it does not contain any ball.
For instance, consider~$\Si=\Si' = [4]$  with signature given by $\Si_0$ (resp.~$\Si'_1$) the set of even numbers and~$\Si_1$ (resp.~$\Si'_0$) defined consequently. The rectangular array in Figure~\ref{SignedBallArray} is a signed ball array on~$\Si$ and~$\Si'$.
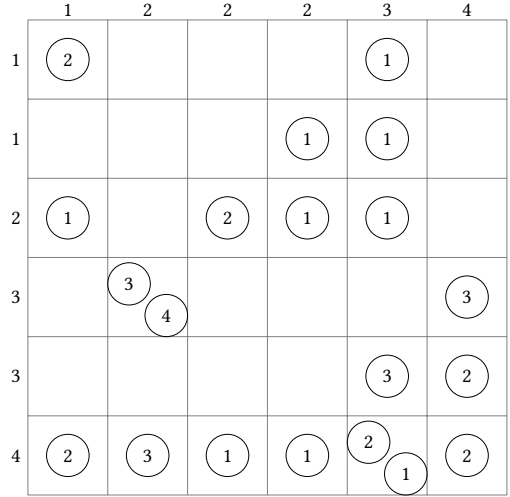
\begin{figure}[h!]
\centering
\scalebox{0.7}{
\begin{tikzpicture}
\draw [very thin, gray] (0,0) grid[step=1.5] (9,9);
\draw (0.75,9.45) node[below] {$\footnotesize 1$};
\draw (2.25, 9.45) node[below] {$\footnotesize 2$};
\draw (3.75, 9.45) node[below] {$\footnotesize 2$};
\draw (5.25, 9.45) node[below] {$\footnotesize 2$};
\draw (6.75, 9.45) node[below] {$\footnotesize 3$};
\draw (8.25, 9.45) node[below] {$\footnotesize 4$};
\draw (0,8.25) node[left] {$\footnotesize 1$};
\draw (0,6.75) node[left] {$\footnotesize 1$};
\draw (0,5.25) node[left] {$\footnotesize 2$};
\draw (0,3.75) node[left] {$\footnotesize 3$};
\draw (0,2.25) node[left] {$\footnotesize 3$};
\draw (0,0.75) node[left] {$\footnotesize 4$};
\draw (0.75,8.25) circle (0.4);
\draw (0.75,8.25) node {$\footnotesize 2$};
\draw (0.75,5.25) circle (0.4);
\draw (0.75,5.25) node {$\footnotesize 1$};
\draw (0.75,0.75) circle (0.4);
\draw (0.75,0.75) node {$\footnotesize 2$};
\draw (2.25,0.75) circle (0.4);
\draw (2.25,0.75) node {$\footnotesize 3$};
\draw (1.9,4) circle (0.4);
\draw (1.9,4) node {$\footnotesize 3$};
\draw (2.6,3.4) circle (0.4);
\draw (2.6,3.4) node {$\footnotesize 4$};
\draw  (3.75,0.75) circle (0.4) ;
\draw (3.75,0.75) node {$\footnotesize 1$};
\draw (3.75,5.25) circle (0.4);
\draw (3.75,5.25) node {$\footnotesize 2$};
\draw (5.25,6.75) circle (0.4);
\draw (5.25,6.75) node {$\footnotesize 1$};
\draw (5.25,5.25) circle (0.4);
\draw (5.25,5.25) node {$\footnotesize 1$};
\draw (5.25,0.75) circle (0.4) ;
\draw  (5.25,0.75) node {$\footnotesize 1$};
\draw (6.75,8.25) circle (0.4) ;
\draw (6.75,8.25) node {$\footnotesize 1$};
\draw (6.75,6.75) circle (0.4) ;
\draw  (6.75,6.75) node {$\footnotesize 1$};
\draw (6.75,5.25) circle (0.4) ;
\draw  (6.75,5.25) node {$\footnotesize 1$};
\draw (6.75,2.25) circle (0.4) ;
 \draw  (6.75,2.25) node {$\footnotesize 3$};
  \draw (6.4,1) circle (0.4) ;
 \draw  (6.4,1) node {$\footnotesize 2$};
  \draw (7.1,0.4) circle (0.4) ;
 \draw  (7.1,0.4) node {$\footnotesize 1$};
 \draw (8.25,3.75) circle (0.4) ;
 \draw  (8.25,3.75) node {$\footnotesize 3$};
  \draw (8.25,2.25) circle (0.4) ;
 \draw  (8.25,2.25) node {$\footnotesize 2$};
  \draw (8.25,0.75) circle (0.4) ;
 \draw  (8.25,0.75) node {$\footnotesize 2$};
\end{tikzpicture}
}
\caption{Example of a signed ball array.}
\label{SignedBallArray}
\end{figure}

We will use the following notations to describe the relative positions of two boxes in a signed ball array.
A box $b'$ is \emph{West} (resp. \emph{west}) of a box~$b$ if the column of~$b'$ is strictly to the left of (resp. left of or equal to) the column of~$b$.
A box $b'$ is \emph{North} (resp. \emph{north}) of a box~$b$ if the row of~$b'$ is strictly above (resp. above or equal to) the row of~$b$.
Similarly, we define the other positions corresponding to the east and south directions using capital and small letters to denote strict and weak positions.
A box $b'$ is \emph{Northwest} of a box~$b$ if the row of $b'$ is strictly above to the row of $b$, and the column of~$b'$ is left or equal to the column of $b$.
Similarly, we define the other combinations of positions  corresponding to the four cardinal directions using capital and small letters to denote strict and weak inequalities.

\subsubsection{The matrix-ball construction}
\label{SSS:MatrixBallConstruction}
Let~$\Si$ and~$\Si'$ be  signed alphabets. We will associate to each signed two-rowed array~$w= \begin{pmatrix}
 x_1&\ldots &x_k\\
 y_1&\ldots &y_k
\end{pmatrix}$  on~$\Si$ and~$\Si'$,   a signed ball array, denoted by~$\Ba(w)$, whose rows (resp. columns) are indexed with the elements of the first (resp. second) row of~$w$ and where the indices from~$\Si_1$ and~$\Si'_1$ are repeated as many times as they appear in~$w$, as described in the following three steps:
\medskip

\noindent {\bf Step~1}. 
We start with an empty signed ball array, whose rows (resp. columns) are indexed with the elements of the first (resp. second) row of~$w$ and where the indices from~$\Si_1$ and~$\Si'_1$ are repeated as many times as they appear in~$w$. 
We then associate to each couple~$(x_i, y_i)$ in~$w$ for~$i=1,\ldots k$, a ball in a box of the initial empty signed ball array according to the following four cases:
\begin{enumerate}[\bf i)]
\item Suppose~$(x_i,y_i)\in\Si_0\times \Si'_0$.
Following Condition~\eqref{E:Order2}, equal couples of this form can occur in~$w$. For each couple~$(x_i, y_i)$ we associate a ball in the box corresponding to the row indexed with~$x_i$ and to the column indexed with~$y_i$ in the signed ball array. In this case, the number of balls in the same position~$(x_i, y_i)$ is equal to the multiplicity of the couple~$(x_i, y_i)$ in~$w$.
\item Suppose~$(x_i, y_i)\in \Si_1\times \Si'_1$. 
Following Condition~\eqref{E:Order2},  equal couples of this form can occur in~$w$.
For each couple~$(x_i, y_i)$ we associate a ball in the empty box of the signed ball array corresponding to the topmost row indexed with~$x_i$ and the rightmost column indexed with~$y_i$ such that only one ball can occur in the same position, and if there are many rows (rep. columns) indexed with~$x_i$ (resp.~$y_i$) we choose the topmost (resp. rightmost) row (resp. column) that does not contain any ball.
\item Suppose~$(x_i, y_i)\in\Si_0\times \Si'_1$.
Following Condition~\eqref{E:Order2}, we can not have equal couples of this form in~$w$. We add a ball in the empty box of the signed ball array that corresponds to the row indexed with~$x_i$ and to the rightmost column indexed with~$y_i$, and if there are many columns that are indexed with~$y_i$  we choose the rightmost one that does not contain any ball.
\item Suppose~$(x_i, y_i)\in\Si_1\times \Si'_0$. Following Condition~\eqref{E:Order2}, we can not have equal couples of this form in~$w$. We add a ball in the empty box of the signed ball array that corresponds to the topmost row indexed with~$x_i$ and to the column indexed with~$y_i$, and if there are many rows that are indexed with~$x_i$, we choose the topmost one that does not contain any ball.
\end{enumerate}

\noindent {\bf Step~2}.
If many balls occur in the same position, then we order them arbitrarily by arranging them diagonally from NorthWest to SouthEast. A ball is \emph{northwest} of another one if it is in the same position and NorthWest in this arrangement, or its row and column positions are less than or equal to those of the second ball with at least one inequality strict. The \emph{bottom-rightmost} ball of a signed ball array is the ball in the position that corresponds to the bottom-most row and to the rightmost column of the given signed ball array, and if many balls occur in this position it corresponds to the last one to the southeast in the corresponding diagonal arrangement. Similarly, we define the \emph{top-leftmost ball} of a signed ball array.

Working from the top-leftmost ball to the bottom-rightmost ball and starting with the first row of the resulting signed ball array, we number all the balls with positive integers by filling each ball by the smallest integer that is larger than all the integers occurring in the balls to the northwest, such that the balls in the same position are numbered with consecutive integers. More precisely, a ball is numbered  with~$1$ if there are no balls northwest of it. 
A ball is numbered with a positive integer~$i$ if the preceding ball in the same position is numbered with the integer~$i-1$, or if the ball is the first one in a given position and the largest number occurring in a ball northwest of the given position is the integer~$i-1$. The resulting signed ball array is denoted by~$\Ba^1(w)$.

For instance, Figure~\ref{SignedBallArray1} (resp. Figure~\ref{SignedBallArray2})  represents  the signed ball array~$\Ba^1(w)$ corresponding to the signed two-rowed array~$w$ of Example~\ref{Ex:RSKExample1} (resp.~Example~\ref{Ex:RSKExample2}).

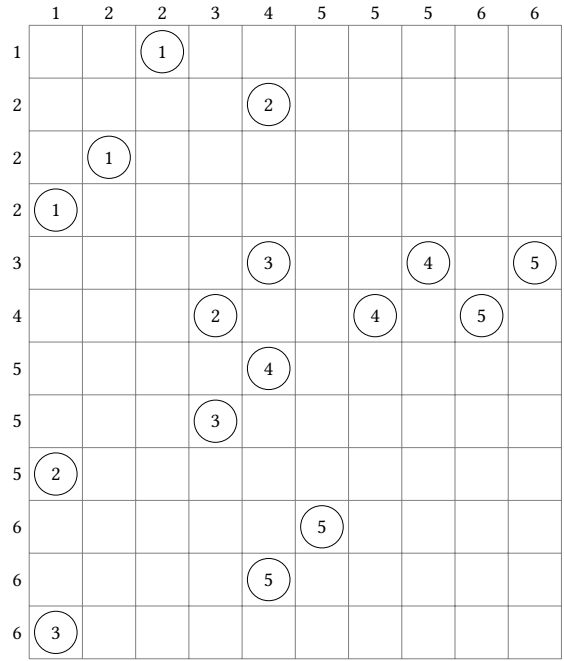
\begin{figure}[h!]
\centering
\scalebox{0.7}{
\begin{tikzpicture}
\draw [very thin, gray] (0,0) grid[step=1] (10,12);
\draw (0.5,12.5) node[below] {$\footnotesize 1$};
\draw (1.5,12.5) node[below] {$\footnotesize 2$};
\draw (2.5,12.5) node[below] {$\footnotesize 2$};
\draw (3.5,12.5) node[below] {$\footnotesize 3$};
\draw (4.5,12.5) node[below] {$\footnotesize 4$};
\draw (5.5,12.5) node[below] {$\footnotesize 5$};
\draw (6.5,12.5) node[below] {$\footnotesize 5$};
\draw (7.5,12.5) node[below] {$\footnotesize 5$};
\draw (8.5,12.5) node[below] {$\footnotesize 6$};
\draw (9.5,12.5) node[below] {$\footnotesize 6$};
\draw (0,11.5) node[left] {$\footnotesize 1$};
\draw (0,10.5) node[left] {$\footnotesize 2$};
\draw (0,9.5) node[left] {$\footnotesize 2$};
\draw (0,8.5) node[left] {$\footnotesize 2$};
\draw (0,7.5) node[left] {$\footnotesize 3$};
\draw (0,6.5) node[left] {$\footnotesize 4$};
\draw (0,5.5) node[left] {$\footnotesize 5$};
\draw (0,4.5) node[left] {$\footnotesize 5$};
\draw (0,3.5) node[left] {$\footnotesize 5$};
\draw (0,2.5) node[left] {$\footnotesize 6$};
\draw (0,1.5) node[left] {$\footnotesize 6$};
\draw (0,0.5) node[left] {$\footnotesize 6$};
  \draw (0.5,8.5) circle (0.4) ;
 \draw (0.5,8.5) node {$\footnotesize 1$};
 \draw (0.5,3.5) circle (0.4) ;
 \draw (0.5,3.5) node {$\footnotesize 2$};
  \draw (0.5,0.5) circle (0.4) ;
 \draw (0.5,0.5) node {$\footnotesize 3$};
 \draw (1.5,9.5) circle (0.4) ;
 \draw (1.5,9.5) node {$\footnotesize 1$};
 \draw (2.5,11.5) circle (0.4) ;
 \draw (2.5,11.5) node {$\footnotesize 1$};
 \draw (3.5,6.5) circle (0.4) ;
 \draw (3.5,6.5) node {$\footnotesize 2$};
 \draw (3.5,4.5) circle (0.4) ;
 \draw (3.5,4.5) node {$\footnotesize 3$};
 \draw (4.5,10.5) circle (0.4) ;
 \draw (4.5,10.5) node {$\footnotesize 2$};
 \draw (4.5,7.5) circle (0.4) ;
 \draw (4.5,7.5) node {$\footnotesize 3$};
 \draw (4.5,5.5) circle (0.4) ;
 \draw (4.5,5.5) node {$\footnotesize 4$};
  \draw (4.5,1.5) circle (0.4) ;
 \draw (4.5,1.5) node {$\footnotesize 5$};
\draw (5.5,2.5) circle (0.4) ;
 \draw (5.5,2.5) node {$\footnotesize 5$};
 \draw (6.5,6.5) circle (0.4) ;
 \draw (6.5,6.5) node {$\footnotesize 4$};
 \draw (7.5,7.5) circle (0.4) ;
 \draw (7.5,7.5) node {$\footnotesize 4$};
 \draw (8.5,6.5) circle (0.4) ;
 \draw (8.5,6.5) node {$\footnotesize 5$};
 \draw (9.5,7.5) circle (0.4) ;
 \draw (9.5,7.5) node {$\footnotesize 5$};
\end{tikzpicture}
}
\caption{$\Ba^1(w)$ corresponding to~$w$ of Example~\ref{Ex:RSKExample1}.}
\label{SignedBallArray1}
\end{figure}

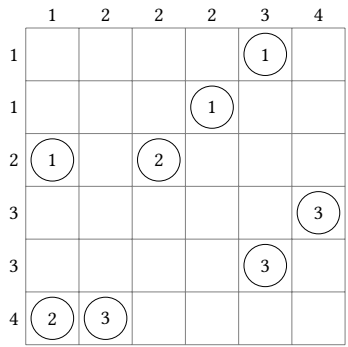
\begin{figure}[h!]
\centering
\scalebox{0.7}{
\begin{tikzpicture}
\draw [very thin, gray] (0,0) grid[step=1] (6,6);
\draw (0.5, 6.5) node[below] {$\footnotesize 1$};
\draw (1.5, 6.5) node[below] {$\footnotesize 2$};
\draw (2.5, 6.5) node[below] {$\footnotesize 2$};
\draw (3.5, 6.5) node[below] {$\footnotesize 2$};
\draw (4.5, 6.5) node[below] {$\footnotesize 3$};
\draw (5.5, 6.5) node[below] {$\footnotesize 4$};
\draw (0,5.5) node[left] {$\footnotesize 1$};
\draw (0,4.5) node[left] {$\footnotesize 1$};
\draw (0,3.5) node[left] {$\footnotesize 2$};
\draw (0,2.5) node[left] {$\footnotesize 3$};
\draw (0,1.5) node[left] {$\footnotesize 3$};
\draw (0,0.5) node[left] {$\footnotesize 4$};
  \draw (0.5,3.5) circle (0.4);
 \draw (0.5,3.5) node {$\footnotesize 1$};
  \draw (0.5,0.5) circle (0.4);
 \draw (0.5,0.5) node {$\footnotesize 2$};
\draw (1.5,0.5) circle (0.4);
 \draw (1.5,0.5) node {$\footnotesize 3$};
 \draw (2.5,3.5) circle (0.4);
 \draw (2.5,3.5) node {$\footnotesize 2$};
\draw (3.5,4.5) circle (0.4);
 \draw (3.5,4.5) node {$\footnotesize 1$};
 \draw (4.5,5.5) circle (0.4) ;
 \draw (4.5,5.5) node {$\footnotesize 1$};
\draw (4.5,1.5) circle (0.4) ;
 \draw  (4.5,1.5) node {$\footnotesize 3$};
 \draw (5.5,2.5) circle (0.4) ;
 \draw  (5.5,2.5) node {$\footnotesize 3$};
\end{tikzpicture}
}
\caption{$\Ba^1(w)$ corresponding to~$w$ of Example~\ref{Ex:RSKExample2}.}
\label{SignedBallArray2}
\end{figure}

\noindent {\bf Step~3}.
If there are~$k>1$ balls filled with the same integer~$i$ in the resulting signed ball array~$\Ba^1(w)$, then they belong by construction to a string from SouthWest to NorthEast. We then introduce new~$k-1$ balls by putting a ball to the right of each ball in the string but the last, directly under the next ball. 
We will use a new color for the added balls as illustrated in Figure~\ref{SignedBallArray3}.
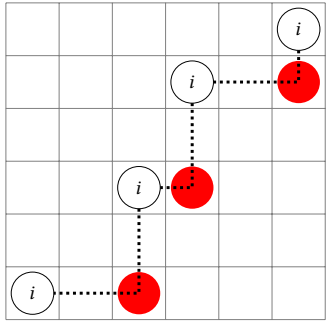
\begin{figure}[h!]
\centering
\scalebox{0.7}{
\begin{tikzpicture}
\draw [very thin, gray] (0,0) grid[step=1] (6,6);
\draw (0.5,0.5) circle (0.4);
\draw (0.5,0.5) node {$\footnotesize i$};
\draw (2.5,2.5) circle (0.4);
\draw (2.5,2.5) node {$\footnotesize i$};
\fill [red] (2.5,0.5) circle (0.4);
\fill [red] (3.5,2.5) circle (0.4);
 \draw (3.5,4.5) circle (0.4);
 \draw (3.5,4.5) node {$\footnotesize i$};
\draw (5.5,5.5) circle (0.4);
 \draw (5.5,5.5) node {$\footnotesize i$};
\fill [red] (5.5,4.5) circle (0.4);
\draw[ultra thick, dotted] (5.5,5.1)  -- (5.5,4.5)-- (3.9,4.5);
\draw[ultra thick, dotted] (3.5,4.1) --  (3.5,2.5) --  (2.9,2.5);
\draw[ultra thick, dotted]  (2.5, 2.1) --(2.5, 0.5)--(0.9,0.5);
\end{tikzpicture}
}
\caption{String for new colored balls}
\label{SignedBallArray3}
\end{figure}
We do the same for all the balls filled with the same integer.
We then number the new added balls as per~{\bf Step~2} without taking into consideration the numbering of the initial non-colored balls and by just acting on the new balls from the top-leftmost ball to the bottom-rightmost one. We obtain a new signed ball array, denoted by~$\Ba^2(w)$, that contains the initial non-colored balls and the new colored ones.
We repeat the same process on the new added balls and we construct~$\Ba^3(w)$ from~$\Ba^2(w)$ by adding new colored balls and by numbering them, and so on, stopping when no two balls appear in~$\Ba^k(w)$ for any~$k>1$ with the same number. The resulting signed ball array is denoted by~$\Ba(w)$. Note that we will use the same color for all the colored balls added to compute~$\Ba^k(w)$ from~$\Ba^{k-1}(w)$ for any~$k>1$.
We  denote by~$\Ba_{0}(w)$ the signed ball array obtained from~$\Ba(w)$ by eliminating the initial non colored balls and by keeping only the new colored ones, and  by~$\Ba^{1}_{0}(w)$ the signed ball array obtained from~$\Ba_{0}(w)$ by keeping only the colored balls added to compute~$\Ba^2(w)$ from~$\Ba^1(w)$ and by eliminating all the other colored balls. 
For instance, Figure~\ref{SignedBallArray4} (resp. Figure~\ref{SignedBallArray5} ) represents the signed ball array~$\Ba(w)$ corresponding to the signed two-rowed array~$w$ of Example~\ref{Ex:RSKExample1} (resp. Example~\ref{Ex:RSKExample2}).

\begin{figure}[h!]
\centering
\scalebox{0.7}{
\begin{tikzpicture}
\draw [very thin, gray] (0,0) grid[step=1] (10,12);
\draw (0.5,12.5) node[below] {$\footnotesize 1$};
\draw (1.5,12.5) node[below] {$\footnotesize 2$};
\draw (2.5,12.5) node[below] {$\footnotesize 2$};
\draw (3.5,12.5) node[below] {$\footnotesize 3$};
\draw (4.5,12.5) node[below] {$\footnotesize 4$};
\draw (5.5,12.5) node[below] {$\footnotesize 5$};
\draw (6.5,12.5) node[below] {$\footnotesize 5$};
\draw (7.5,12.5) node[below] {$\footnotesize 5$};
\draw (8.5,12.5) node[below] {$\footnotesize 6$};
\draw (9.5,12.5) node[below] {$\footnotesize 6$};
\draw (0,11.5) node[left] {$\footnotesize 1$};
\draw (0,10.5) node[left] {$\footnotesize 2$};
\draw (0,9.5) node[left] {$\footnotesize 2$};
\draw (0,8.5) node[left] {$\footnotesize 2$};
\draw (0,7.5) node[left] {$\footnotesize 3$};
\draw (0,6.5) node[left] {$\footnotesize 4$};
\draw (0,5.5) node[left] {$\footnotesize 5$};
\draw (0,4.5) node[left] {$\footnotesize 5$};
\draw (0,3.5) node[left] {$\footnotesize 5$};
\draw (0,2.5) node[left] {$\footnotesize 6$};
\draw (0,1.5) node[left] {$\footnotesize 6$};
\draw (0,0.5) node[left] {$\footnotesize 6$};
  \draw (0.5,8.5) circle (0.4) ;
 \draw (0.5,8.5) node {$\footnotesize 1$};
 \draw (0.5,3.5) circle (0.4) ;
 \draw (0.5,3.5) node {$\footnotesize 2$};
  \draw (0.5,0.5) circle (0.4) ;
 \draw (0.5,0.5) node {$\footnotesize 3$};
 \draw (1.5,9.5) circle (0.4) ;
 \draw (1.5,9.5) node {$\footnotesize 1$};
 \fill   [red] (1.5,8.5) circle (0.4) ;
 \draw  (1.5,8.5) node {$\footnotesize 1$};
 \draw (2.5,11.5) circle (0.4) ;
 \draw (2.5,11.5) node {$\footnotesize 1$};
 \fill   [red] (2.5,9.5) circle (0.4) ;
 \draw  (2.5,9.5) node {$\footnotesize 1$};
 \fill   [green] (2.5,8.5) circle (0.4) ;
 \draw  (2.5,8.5) node {$\footnotesize 1$};
 \draw (3.5,6.5) circle (0.4) ;
 \draw (3.5,6.5) node {$\footnotesize 2$};
 \draw (3.5,4.5) circle (0.4) ;
 \draw (3.5,4.5) node {$\footnotesize 3$};
 \fill   [red] (3.5,3.5) circle (0.4) ;
 \draw (3.5,3.5) node {$\footnotesize 2$};
 \fill   [red] (3.5,0.5) circle (0.4) ;
 \draw  (3.5,0.5) node {$\footnotesize 3$};
 \draw (4.5,10.5) circle (0.4) ;
 \draw (4.5,10.5) node {$\footnotesize 2$};
 \draw (4.5,7.5) circle (0.4) ;
 \draw (4.5,7.5) node {$\footnotesize 3$};
 \draw (4.5,5.5) circle (0.4) ;
 \draw (4.5,5.5) node {$\footnotesize 4$};
  \draw (4.5,1.5) circle (0.4) ;
 \draw (4.5,1.5) node {$\footnotesize 5$};
  \fill   [red] (4.5,6.5) circle (0.4) ;
 \draw (4.5,6.5) node {$\footnotesize 2$};
 \fill   [red] (4.5,4.5) circle (0.4) ;
 \draw  (4.5,4.5) node {$\footnotesize 3$};
  \fill   [green] (4.5,3.5) circle (0.4) ;
 \draw  (4.5,3.5) node {$\footnotesize 2$};
 \fill   [green] (4.5,0.5) circle (0.4) ;
 \draw  (4.5,0.5) node {$\footnotesize 3$};
\draw (5.5,2.5) circle (0.4) ;
 \draw (5.5,2.5) node {$\footnotesize 5$};
\fill   [red] (5.5,1.5) circle (0.4) ;
 \draw  (5.5,1.5) node {$\footnotesize 4$};
 \draw (6.5,6.5) circle (0.4) ;
 \draw (6.5,6.5) node {$\footnotesize 4$};
 \fill   [red] (6.5,5.5) circle (0.4) ;
 \draw  (6.5,5.5) node {$\footnotesize 3$};
 \fill   [green] (6.5,4.5) circle (0.4) ;
 \draw  (6.5,4.5) node {$\footnotesize 2$};
  \fill   [yellow] (6.5,3.5) circle (0.4) ;
 \draw  (6.5,3.5) node {$\footnotesize 1$};
 \draw (7.5,7.5) circle (0.4) ;
 \draw (7.5,7.5) node {$\footnotesize 4$};
 \fill   [red] (7.5,6.5) circle (0.4) ;
 \draw  (7.5,6.5) node {$\footnotesize 3$};
 \fill   [green] (7.5,5.5) circle (0.4) ;
 \draw  (7.5,5.5) node {$\footnotesize 2$};
 \fill   [yellow] (7.5,4.5) circle (0.4) ;
 \draw  (7.5,4.5) node {$\footnotesize 1$};
 \fill   [BlueD] (7.5,3.5) circle (0.4) ;
 \draw  (7.5,3.5) node {$\footnotesize 1$};
 \draw (8.5,6.5) circle (0.4) ;
 \draw (8.5,6.5) node {$\footnotesize 5$};
 \fill   [red] (8.5,2.5) circle (0.4) ;
 \draw  (8.5,2.5) node {$\footnotesize 4$};
 \fill   [green] (8.5,1.5) circle (0.4) ;
 \draw  (8.5,1.5) node {$\footnotesize 3$};
 \fill   [yellow] (8.5,0.5) circle (0.4) ;
 \draw  (8.5,0.5) node {$\footnotesize 2$};
 \draw (9.5,7.5) circle (0.4) ;
 \draw (9.5,7.5) node {$\footnotesize 5$};
 \fill   [red] (9.5,6.5) circle (0.4) ;
   \draw (9.5,6.5) node {$\footnotesize 4$};
   \fill   [green] (9.5,2.5) circle (0.4) ;
 \draw  (9.5,2.5) node {$\footnotesize 3$};
  \fill   [yellow] (9.5,1.5) circle (0.4) ;
 \draw  (9.5,1.5) node {$\footnotesize 2$};
\fill   [BlueD] (9.5,0.5) circle (0.4) ;
 \draw  (9.5,0.5) node {$\footnotesize 2$};
\end{tikzpicture}
}
\caption{$\Ba(w)$ corresponding to~$w$ of Example~\ref{Ex:RSKExample1}.}
\label{SignedBallArray4}
\end{figure}

\begin{figure}[h!]
\centering
\scalebox{0.7}{
\begin{tikzpicture}
\draw [very thin, gray] (0,0) grid[step=1.5] (9,9);
\draw (0.75,9.45) node[below] {$\footnotesize 1$};
\draw (2.25, 9.45) node[below] {$\footnotesize 2$};
\draw (3.75, 9.45) node[below] {$\footnotesize 2$};
\draw (5.25, 9.45) node[below] {$\footnotesize 2$};
\draw (6.75, 9.45) node[below] {$\footnotesize 3$};
\draw (8.25, 9.45) node[below] {$\footnotesize 4$};
\draw (0,8.25) node[left] {$\footnotesize 1$};
\draw (0,6.75) node[left] {$\footnotesize 1$};
\draw (0,5.25) node[left] {$\footnotesize 2$};
\draw (0,3.75) node[left] {$\footnotesize 3$};
\draw (0,2.25) node[left] {$\footnotesize 3$};
\draw (0,0.75) node[left] {$\footnotesize 4$};
  \draw (0.75,5.25) circle (0.4);
 \draw (0.75,5.25) node {$\footnotesize 1$};
  \draw (0.75,0.75) circle (0.4);
 \draw (0.75,0.75) node {$\footnotesize 2$};
\draw (2.25,0.75) circle (0.4);
 \draw (2.25,0.75) node {$\footnotesize 3$};
 \fill [red]  (3.75,0.75) circle (0.4) ;
 \draw (3.75,0.75) node {$\footnotesize 1$};
 \draw (3.75,5.25) circle (0.4);
 \draw (3.75,5.25) node {$\footnotesize 2$};
\draw (5.25,6.75) circle (0.4);
 \draw (5.25,6.75) node {$\footnotesize 1$};
\fill [red] (5.25,5.25) circle (0.4);
 \draw (5.25,5.25) node {$\footnotesize 1$};
\fill   [green] (5.25,0.75) circle (0.4) ;
 \draw  (5.25,0.75) node {$\footnotesize 1$};
 \draw (6.75,8.25) circle (0.4) ;
 \draw (6.75,8.25) node {$\footnotesize 1$};
 \fill   [red] (6.75,6.75) circle (0.4) ;
 \draw  (6.75,6.75) node {$\footnotesize 1$};
 \fill   [green] (6.75,5.25) circle (0.4) ;
 \draw  (6.75,5.25) node {$\footnotesize 1$};
\draw (6.75,2.25) circle (0.4) ;
 \draw  (6.75,2.25) node {$\footnotesize 3$};
 \fill   [red] (6.4,1) circle (0.4) ;
 \draw  (6.4,1) node {$\footnotesize 2$};
  \fill   [yellow] (7.1,0.4) circle (0.4) ;
 \draw  (7.1,0.4) node {$\footnotesize 1$};
 \draw (8.25,3.75) circle (0.4) ;
 \draw  (8.25,3.75) node {$\footnotesize 3$};
 \fill   [red] (8.25,2.25) circle (0.4) ;
 \draw  (8.25,2.25) node {$\footnotesize 2$};
 \fill   [green] (8.25,0.75) circle (0.4) ;
 \draw  (8.25,0.75) node {$\footnotesize 2$};
\end{tikzpicture}
}
\caption{$\Ba(w)$ corresponding to~$w$ of Example~\ref{Ex:RSKExample2}.}
\label{SignedBallArray5}
\end{figure}

\subsubsection{Super tableaux for signed ball arrays}
\label{SSS:TableauxSignedBallArray}

Let~$w$ be a signed two-rowed array on  signed alphabets~$\Si$ and~$\Si'$.
We denote by~$\T(\Ba(w))$ (resp.~$\Q(\Ba(w)))$  the super tableau obtained from $\Ba(w)$ such that its~$k$-th row lists the indices of the leftmost columns (resp. top-most rows) of~$\Ba^k(w)$ where each integer number occurs in the new added balls.  That is, the $i$-th entry of the first row of~$\T(\Ba(w))$ (resp.~$\Q(\Ba(w)))$) is the index of the leftmost column (resp. top-most row) in~$\Ba^1(w)$ where a ball filled with~$i$ occurs, and the $i$-th entry of the $k$-th row of~$\T(\Ba(w))$ (resp.~$\Q(\Ba(w)))$, for~$k>1$, is equal to the index of the leftmost column (resp. top-most row) in~$\Ba^k(w)$ where a new colored ball filled with~$i$ occurs.

\begin{proposition}
\label{P:MatrixBallConstruction}
Let~$w$ be a signed two-rowed array on signed alphabets~$\Si$ and~$\Si'$. The following equality~$(\T(\Ba(w)), \Q(\Ba(w))) =(\T(w), \Q(w))$ holds. 
\end{proposition}

Theorem~\ref{T:SymmetryRSK} is then a direct consequence of this proposition, since the matrix-ball construction is symmetric in the rows and columns of the resulting signed ball array. In the rest of this subsection we will prove Proposition~\ref{P:MatrixBallConstruction}. Before that, we give the following example that illustrates the symmetry property of the super-RSK correspondence for super tableaux.

\begin{example}
\label{Ex:MatrixBallConstruction}
Consider the signed two-rowed array~$w$ of Example~\ref{Ex:RSKExample1}. 
The super tableaux associated to its signed ball array~$\Ba(w)$ illustrated in Figure~\ref{SignedBallArray4} are the following:
\[
\raisebox{0.5cm}{
\raisebox{-0.5cm}{$\T(\Ba(w))\; =\; $}
{\ytableausetup{mathmode, boxsize=1em}
\begin{ytableau}
1&1&1&4&4\\
2&3&3&5\\
2&4&4\\
5&6\\
5&6
\end{ytableau}}}
\;
,
\;
\raisebox{0.5cm}{
\raisebox{-0.5cm}{$\Q(\Ba(w))\; =\; $}
{\ytableausetup{mathmode, boxsize=1em}
\begin{ytableau}
1&2&3&3&3\\
2&4&4&4\\
2&5&6\\
5&6\\
5&6
\end{ytableau}}
}
\]
which are equal to the ones of~$w$ already computed in Example~\ref{Ex:RSKExample1}:
\[
w\;=\;
\setcounter{MaxMatrixCols}{20}
\begin{pmatrix}
 1&2&2&2&3&3&3&4&4&4&5&5&5&6&6&6\\
 2&4&2&1&4&5&6&3&5&6&4&3&1&5&4&1
\end{pmatrix}
\;
\overset{\RSK}{\longleftrightarrow}
\;
\Bigg(\raisebox{0.5cm}{
\raisebox{-0.5cm}{$\T(w)\; =\; $}
{\ytableausetup{mathmode, boxsize=1em}
\begin{ytableau}
1&1&1&4&4\\
2&3&3&5\\
2&4&4\\
5&6\\
5&6
\end{ytableau}}}
\;
,
\;
\raisebox{0.5cm}{
\raisebox{-0.5cm}{$\Q(w)\; =\; $}
{\ytableausetup{mathmode, boxsize=1em}
\begin{ytableau}
1&2&3&3&3\\
2&4&4&4\\
2&5&6\\
5&6\\
5&6
\end{ytableau}}
}\Bigg).
\]
Moreover, using the matrix-ball construction, we obtain:
\[
w^{\text{inv}}\;=\;
\setcounter{MaxMatrixCols}{20}
\begin{pmatrix}
 1&1&1&2&2&3&3&4&4&4&4&5&5&5&6&6\\
 2&5&6&2&1&4&5&2&3&5&6&6&4&3&4&3
\end{pmatrix}
\;
\overset{\RSK}{\longleftrightarrow}
\;
\Bigg(\raisebox{0.5cm}{
\raisebox{-0.5cm}{$\Q(w)\; =\; $}
{\ytableausetup{mathmode, boxsize=1em}
\begin{ytableau}
1&2&3&3&3\\
2&4&4&4\\
2&5&6\\
5&6\\
5&6
\end{ytableau}}}
\;
,
\;
\raisebox{0.5cm}{
\raisebox{-0.5cm}{$\T(w)\; =\; $}
{\ytableausetup{mathmode, boxsize=1em}
\begin{ytableau}
1&1&1&4&4\\
2&3&3&5\\
2&4&4\\
5&6\\
5&6
\end{ytableau}}
}\Bigg)
\]
by switching the roles of the rows and the columns of the signed ball array of Figure~\ref{SignedBallArray4}.

Consider now the signed two-rowed array~$w$ of Example~\ref{Ex:RSKExample2}. The super tableaux associated to its signed ball array~$\Ba(w)$ illustrated in Figure~\ref{SignedBallArray5} are the following:
\[
\raisebox{0.5cm}{
\raisebox{-0.5cm}{$\T(\Ba(w))\; =\; $}
{\ytableausetup{mathmode, boxsize=1em}
\begin{ytableau}
1&1&2\\
2&3\\
2&4\\
3
\end{ytableau}}}
\;
,
\;
\raisebox{0.5cm}{
\raisebox{-0.5cm}{$\Q(\Ba(w))\; =\; $}
{\ytableausetup{mathmode, boxsize=1em}
\begin{ytableau}
1&2&3\\
1&3\\
2&4\\
4
\end{ytableau}}
}
\]
which are equal to the ones of~$w$ already computed in Example~\ref{Ex:RSKExample2}:
\[
w\;=\;
\setcounter{MaxMatrixCols}{20}
\begin{pmatrix}
 1&1&2&2&3&3&4&4\\
 3&2&1&2&4&3&1&2
\end{pmatrix}
\quad
\overset{\RSK}{\longleftrightarrow}
\quad
\Bigg(\raisebox{0.5cm}{
\raisebox{-0.5cm}{$\T(w)\; =\; $}
{\ytableausetup{mathmode, boxsize=1em}
\begin{ytableau}
1&1&2\\
2&3\\
2&4\\
3
\end{ytableau}}}
\;
,
\;
\raisebox{0.5cm}{
\raisebox{-0.5cm}{$\Q(w)\; =\; $}
{\ytableausetup{mathmode, boxsize=1em}
\begin{ytableau}
1&2&3\\
1&3\\
2&4\\
4
\end{ytableau}}
}\Bigg).
\]
Moreover, using the super matrix-ball construction, we obtain:
\[
w^{\text{inv}}\;=\;
\setcounter{MaxMatrixCols}{20}
\begin{pmatrix}
 1&1&2&2&2&3&3&4\\
 2&4&4&2&1&1&3&3
\end{pmatrix}
\quad
\overset{\RSK}{\longleftrightarrow}
\quad
\Bigg(\raisebox{0.5cm}{
\raisebox{-0.5cm}{$\Q(w)\; =\; $}
{\ytableausetup{mathmode, boxsize=1em}
\begin{ytableau}
1&2&3\\
1&3\\
2&4\\
4
\end{ytableau}}}
\;
,
\;
\raisebox{0.5cm}{
\raisebox{-0.5cm}{$\T(w)\; =\; $}
{\ytableausetup{mathmode, boxsize=1em}
\begin{ytableau}
1&1&2\\
2&3\\
2&4\\
3
\end{ytableau}}
}\Bigg)
\]
by switching the roles of the rows and the columns of the signed ball array of Figure~\ref{SignedBallArray5}.
\end{example}

\subsubsection{Proof of Proposition~\ref{P:MatrixBallConstruction}}
Let~$w= \begin{pmatrix}
 x_1&\ldots &x_k\\
 y_1&\ldots &y_k
\end{pmatrix}$
 be a signed two-rowed array on signed alphabets~$\Si$ and~$\Si'$. We show the result by induction on the number of couples in~$w$, which is equal to the number of balls in~$\Ba^1(w)$.
The result is obvious when~$w$ contains zero or one couple. Let
\[
 w_0\; :=\;  \begin{pmatrix}
 x_1&\ldots &x_{k-1}\\
 y_1&\ldots &y_{k-1}
\end{pmatrix}
 \]
be the signed two-rowed array obtained from~$w$ by eliminating its rightmost couple~$(x_k, y_k)$. 
By the induction hypothesis, we have~$\T(w_0) = \T(\Ba(w_0))$ and~$\Q(w_0) = \Q(\Ba(w_0))$.
Then it is sufficient to prove  the following property:

\begin{property}
\label{Pr:MatrixBallConstruction}
The equality~$\T(\Ba(w)) \;=\; \T(\Ba(w_0))\insr{}y_k$ holds in~$\YoungT(\Si')$, and the super tableau~$\Q(\Ba(w))$ is obtained from~$\Q(\Ba(w_0))$ by placing~$x_k$ in the box that belongs to~$\T(\Ba(w))$ but not to~$\T(\Ba(w_0))$. 
\end{property}

By construction, the signed ball array~$\Ba^1(w)$ contains one ball that is not in~$\Ba^1(w_0)$,  which we denote by~$B$, and we suppose is filled with an integer~$j$. 
This ball is the bottom-rightmost ball of~$\Ba^1(w)$ which belongs to the position corresponding to the bottom-most row indexed by~$x_k$ and to the rightmost (resp. leftmost) column indexed by~$y_k$ if~$y_k\in\Si_0$ (resp.~$y_k\in\Si_1$). Moreover, the integer~$j$ is  larger than all the numbers of the balls in its position. 

Suppose first that there are no more balls filled with~$j$ in~$\Ba^1(w)$.
In this case, all the remaining balls in~$\Ba^1(w)$  are to the NorthWest of~$B$, and then they are all filled with integers that are strictly smaller than~$j$ .
Hence when computing~$\Ba^i(w)$, for~$i>1$, no colored balls will be added using~$B$, and all the rows of~$\T(w)$ and~$\T(w_0)$ (resp.~$\Q(w)$ and~$\Q(w_0)$) below the top-most row are the same.
Moreover, the first row of~$\T(\Ba(w))$ is obtained from that of~$\T(\Ba(w_0))$ by adding~$y_k$ to the end, since~$j$ is the largest integer in~$\Ba^1(w)$, and then~$\Q(w)$ is obtained from~$\Q(w_0)$ by adding~$x_k$ to the end of its first row, showing Property~\ref{Pr:MatrixBallConstruction}. 

Suppose now that there are other balls in~$\Ba^1(w)$ filled with~$j$.
By construction, all these balls are to the NorthEast of~$B$.
Consider the ball filled with~$j$ in the position corresponding to the row indexed with~$x$ and the column indexed with~$y$ such that~$x\leq x_k$ is maximal with~$x=x_k$ only if~$x\in\Si'_1$ and~$y\geq y_k$ is minimal with~$y=y_k$ only if~$y\in\Si_1$.
When~$y_k$ is inserted in the first row of~$\T(\Ba(w_0))$, the letter~$y$ is then bumped from the~$j$-th box of this row. Indeed, the entries of the first row of~$\T(\Ba(w_0))$ are the indices of the left-most columns of~$\Ba^1(w_0)$ that contains balls filled with~$1, 2, \ldots$.
The first~$j-1$ entries of this row are less than or equal to~$y_k$, but the $j$-th entry is~$y$ which is the minimal index that satisfies~$y\geq y_k$ with~$y=y_k$ only if~$y\in\Si_1$, then the entry~$y$ is bumped from the~$j$-th box when inserting~$y_k$ in the first row of~$\T(\Ba(w_0))$. Hence the top-most  row of~$\T(\Ba(w))$ is the top-most  row of~$\T(\Ba(w_0))\insr{}y_k$.
Moreover, the super tableau formed by the rows of~$\T(\Ba(w))$ below its top-most row is by construction the super tableau~$\T(\Ba_{0}(w))$, and the super tableau formed by the rows of~$\T(\Ba(w_0))$ below its top-most row is the super tableau~$\T(\Ba_{0}(w_0))$, where~$\Ba_{0}(w_0)$ is obtained from~$\Ba(w_0)$ by eliminating the initial non colored balls and by keeping only the new colored ones. Then it is sufficient to prove that the equality~$\T(\Ba_{0}(w)) = \T(\Ba_{0}(w_0))\insr{}y$ holds in~$\YoungT(\Si')$, and the new box from this insertion is the box that belongs to~$\Q(\Ba_{0}(w))$ but not to~$\Q(\Ba_{0}(w_0))$.
This follows by induction from Property~\ref{Pr:MatrixBallConstruction}, 
provided with the fact that~$(x_k,y)$ is the position of the bottom-rightmost ball of the signed ball array $\Ba^{1}_{0}(w)$. Indeed, by construction, there are no entries below the $x_k$-th row of~$\Ba^{1}_{0}(w)$. Moreover, if there is other ball in the $x_k$-th row of~$\Ba^{1}_{0}(w)$, then this ball is added using two balls of~$\Ba^{1}(w)$ filled with an integer~$i<j$. The first ball belongs to the row indexed by~$x_k$, and the second one lies NorthEast of it, but it is NorthWest of the bottom-rightmost ball filled with~$j$, and then this second ball cannot lie in a column indexed larger than~$y$, showing the claim.

\subsection{Super-RSK correspondence and  taquin}
\label{SSS:SuperRSKSuperJeuDeTaquin}
Before we proceed to the construction of the dual super-RSK correspondence, we prove the following result which relates the super-RSK correspondence to the super jeu de taquin. This result will be useful for the sequel. 

\begin{proposition}
\label{P:SuperRSKSuperJeuDeTaquin}
Let~$w= \begin{pmatrix}
 x_1&\ldots &x_k\\
 y_1&\ldots &y_k
\end{pmatrix}$
be a signed two-rowed array on signed alphabets~$\Si$ and~$\Si'$ such that~$\RSK(w) = (\T(w),\Q(w))$, and let~$t$ be any super tableau in~$\YoungT(\Si')$.
If we compute the super tableau~$((\ldots(t \insr{} y_1)  \insr{} \ldots )\insr{} y_k)$ over~$\YoungT(\Si')$, and we add~$x_1,\ldots, x_k$ successively in the new boxes starting with an empty Young diagram of the same shape as~$t$,  then the entries~$x_1,\ldots, x_k$ form a super skew tableau~$S$ such that~$\Rec(S)=\Q(w)$.
\end{proposition}

\begin{proof}
Consider a super tableau~$t'$ over~$\YoungT(\Si^{<})$ of the same shape as~$t$, where~$\Si^{<}$ is a signed alphabet whose elements are all smaller than the ones of~$\Si$. Following Theorem~\ref{T:SuperRSKCorrespondence}, the pair~$(t, t')$  corresponds  to some signed two-rowed array~$\begin{pmatrix}
 x'_1&\ldots &x'_l\\
 y'_1&\ldots &y'_l
\end{pmatrix}$.
Hence, the  signed two-rowed array~$W:=
\begin{pmatrix}
 x'_1&\ldots &x'_l& x_1&\ldots &x_k\\
 y'_1&\ldots &y'_l&y_1&\ldots &y_k
\end{pmatrix}$
corresponds to the pair~$\bigg(\big(t\insr{}\R(\T(w)\big), Q'\bigg)$,
 where~$Q'$ is a  super tableau over~$\YoungT(\Si^{<}\cup \Si)$ whose entries~$x'_1,\ldots , x'_l$ are the ones of~$t'$, and whose entries~$x_1,\ldots , x_k$ are the ones of~$S$.
Consider now the inverse~$W^{\text{inv}}$  of~$W$. Following Theorem~\ref{T:SymmetryRSK}, the signed two-rowed array~$W^{\text{inv}}$ corresponds to the pair~$\bigg(Q', \big(t\insr{}\R(\T(w)\big)\bigg)$, and the signed two-rowed array obtained from~$W^{\text{inv}}$ by eliminating the couples~$(y'_i, x'_i)$ corresponds to the pair~$(\Q(w), \T(w))$. The word of the second row of~$W^{\text{inv}}$ is then super plactic equivalent to the word~$\R(Q')$, and when we remove the~$y'_j$'s  from this word we obtain a word that is super plactic equivalent to~$\R(\Q(w))$. However, by construction of~$Q'$, when we remove the $l$ smallest letters from~$\R(Q')$, we recover the word~$\R(S)$. Hence, we deduce by Lemma~\ref{L:SuperplacticLargestSmallest} that~$\R(S)$ and~$\R(\Q(w))$ are also super plactic equivalent, showing by Property~\ref{P:SuperJeuDeTaquinPlactic} that~$\Rec(S) = \Q(w)$.
\end{proof}

\begin{example}
\label{Ex:SuperRSKSuperJeuDetaquin}
We show in Example~\ref{Ex:RSKExample2} the following computation:
\[
w\;=\;
\setcounter{MaxMatrixCols}{20}
\begin{pmatrix}
 1&1&2&2&3&3&4&4\\
 3&2&1&2&4&3&1&2
\end{pmatrix}
\quad
\overset{\RSK}{\longleftrightarrow}
\quad
\Bigg(\raisebox{0.5cm}{
\raisebox{-0.5cm}{$\T(w)\; =\; $}
{\ytableausetup{mathmode, boxsize=1em}
\begin{ytableau}
1&1&2\\
2&3\\
2&4\\
3
\end{ytableau}}}
\;
,
\;
\raisebox{0.5cm}{
\raisebox{-0.5cm}{$\Q(w)\; =\; $}
{\ytableausetup{mathmode, boxsize=1em}
\begin{ytableau}
1&2&3\\
1&3\\
2&4\\
4
\end{ytableau}}
}\Bigg).
\]
Consider now the super tableau~$t=
\raisebox{0.4cm}{
{\ytableausetup{mathmode, boxsize=1em}
\begin{ytableau}
1&1&1\\
2&3\\
4&5\\
4
\end{ytableau}}}$.
We insert the elements of the second row of~$w$ into~$t$, and we place the elements of the first row of~$w$ in the new added boxes starting from an empty tableau of the same shape as~$t$.
Then we obtain the following:
\[
\ytableausetup{mathmode, boxsize=1em}
\begin{ytableau}
1&1&1&3\\
2&3\\
4&5\\
4
\end{ytableau}
\;
\raisebox{-0.5cm}{$\rightarrow$}
\;
\ytableausetup{mathmode, boxsize=1em}
\begin{ytableau}
\empty&\empty&\empty&1\\
\empty&\empty\\
\empty&\empty\\
\empty\\
\end{ytableau}
\quad
\raisebox{-0.5cm}{$;$}
\quad
\ytableausetup{mathmode, boxsize=1em}
\begin{ytableau}
1&1&1&2\\
2&3&3\\
4&5\\
4
\end{ytableau}
\;
\raisebox{-0.5cm}{$\rightarrow$}
\;
\ytableausetup{mathmode, boxsize=1em}
\begin{ytableau}
\empty&\empty&\empty&1\\
\empty&\empty&1\\
\empty&\empty\\
\empty\\
\end{ytableau}
\quad
\raisebox{-0.5cm}{$;$}
\quad
\ytableausetup{mathmode, boxsize=1em}
\begin{ytableau}
1&1&1&1\\
2&3&3\\
2&5\\
4\\
4
\end{ytableau}
\;
\raisebox{-0.5cm}{$\rightarrow$}
\;
\ytableausetup{mathmode, boxsize=1em}
\begin{ytableau}
\empty&\empty&\empty&1\\
\empty&\empty&1\\
\empty&\empty\\
\empty\\
2
\end{ytableau}
\;
\raisebox{-0.5cm}{$;$}
\]
\[
\ytableausetup{mathmode, boxsize=1em}
\begin{ytableau}
1&1&1&1&2\\
2&3&3\\
2&5\\
4\\
4
\end{ytableau}
\;
\raisebox{-0.5cm}{$\rightarrow$}
\;
\ytableausetup{mathmode, boxsize=1em}
\begin{ytableau}
\empty&\empty&\empty&1&2\\
\empty&\empty&1\\
\empty&\empty\\
\empty\\
2
\end{ytableau}
\quad
\raisebox{-0.5cm}{$;$}
\quad
\ytableausetup{mathmode, boxsize=1em}
\begin{ytableau}
1&1&1&1&2&4\\
2&3&3\\
2&5\\
4\\
4
\end{ytableau}
\;
\raisebox{-0.5cm}{$\rightarrow$}
\;
\ytableausetup{mathmode, boxsize=1em}
\begin{ytableau}
\empty&\empty&\empty&1&2&3\\
\empty&\empty&1\\
\empty&\empty\\
\empty\\
2
\end{ytableau}
\quad
\raisebox{-0.5cm}{$;$}
\]
\[
\ytableausetup{mathmode, boxsize=1em}
\begin{ytableau}
1&1&1&1&2&3\\
2&3&3&4\\
2&5\\
4\\
4
\end{ytableau}
\;
\raisebox{-0.5cm}{$\rightarrow$}
\;
\ytableausetup{mathmode, boxsize=1em}
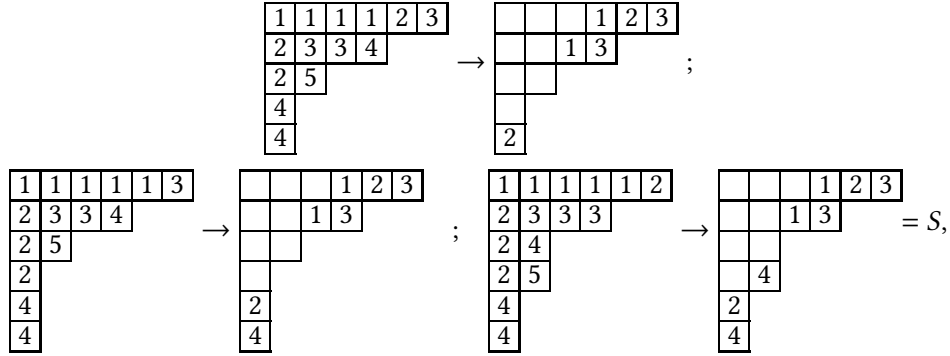
\begin{ytableau}
\empty&\empty&\empty&1&2&3\\
\empty&\empty&1&3\\
\empty&\empty\\
\empty\\
2
\end{ytableau}
\;
\raisebox{-0.5cm}{$;$}
\]
\[
\ytableausetup{mathmode, boxsize=1em}
\begin{ytableau}
1&1&1&1&1&3\\
2&3&3&4\\
2&5\\
2\\
4\\
4
\end{ytableau}
\;
\raisebox{-0.5cm}{$\rightarrow$}
\;
\ytableausetup{mathmode, boxsize=1em}
\begin{ytableau}
\empty&\empty&\empty&1&2&3\\
\empty&\empty&1&3\\
\empty&\empty\\
\empty\\
2\\
4
\end{ytableau}
\quad
\raisebox{-0.5cm}{$;$}
\quad
\ytableausetup{mathmode, boxsize=1em}
\begin{ytableau}
1&1&1&1&1&2\\
2&3&3&3\\
2&4\\
2&5\\
4\\
4
\end{ytableau}
\;
\raisebox{-0.5cm}{$\rightarrow$}
\;
\ytableausetup{mathmode, boxsize=1em}
\begin{ytableau}
\empty&\empty&\empty&1&2&3\\
\empty&\empty&1&3\\
\empty&\empty\\
\empty&4\\
2\\
4
\end{ytableau}
\raisebox{-0.4cm}{$= S,$}
\]
such  that~$\Rec(S) = \Q(w)$, as shown in Example~\ref{Ex:SuperJeuDeTaquin}.
\end{example}

\subsection{A dual super-RSK correspondence}
\label{SS:DualRSKCorrespondence}

We recall the left insertion algorithm on super tableaux which gives us a dual version of the super-RSK correspondence. We first begin by recalling the super evacuation procedure on super tableaux introduced in~\cite{Hage2021Super}.

\subsubsection{Super evacuation on super tableaux}
\label{SSS:SuperEvacuation}
Let~$\Si$ be a signed alphabet.
We denote by~$\Si^{\mathrm{op}}$ the \emph{opposite alphabet} obtained from~$\Si$ by reversing its order,  and by~$x^\ast$ the letter in~$\Si^{\mathrm{op}}$ corresponding to $x$ in~$\Si$ where~$||x^\ast||=1$ (resp.~$||x^\ast||=0$) if~$||x||=1$ (resp.~$||x||=0$). For all~$x, y$ in~$\Si$, we have~$x<y$ if and only if~$x^\ast>y^\ast$. For any word~$w=x_1\ldots x_k$ over~$\Si$, we denote by~$w^\ast=x_k^\ast\ldots x_1^\ast$ the corresponding \emph{opposite word}  of~$w$ over~$\Si^{\mathrm{op}}$.
Then, for all~$v$ and~$w$ in~$\Si^\ast$, the equality~$(vw)^\ast = w^\ast v^\ast$ holds, inducing an anti-isomorphism between the free monoids over~$\Si^{\mathrm{op}}$ and~$\Si$.
By identifying~$(\Si^{\mathrm{op}})^{\mathrm{op}}$ with~$\Si$,  we have~$(w^\ast)^\ast = w$, for any~$w$ in~$\Si^\ast$. Moreover, we show that, for all~$v$ and~$w$ in~$\Si$, the following equivalence holds,~\cite{Hage2021Super}:
\begin{equation}
\label{P:SuperEvacuationPlactic1}
v\sim_{\P(\Si)} w \quad \text{ if, and only,  if }\quad v^\ast\sim_{\P(\Si^{\mathrm{op}})} w^\ast.
\end{equation}

\begin{Algorithm}[\cite{Hage2021Super}]
Let~$t$ be in~$\YoungT(\Si)$.  An \emph{opposite tableau} in~$\YoungT(\Si^{\mathrm{op}})$  can be constructed from~$t$ using the super jeu de taquin, as follows:
~\par\nobreak
\emph{Input:} A super tableau $t$ in~$\YoungT(\Si)$.

\emph{Output:} A super tableau~$t^{\mathrm{op}}$ in~$\YoungT(\Si^{\mathrm{op}})$ with the same shape as~$t$.

\emph{Method:} Start with an empty Young diagram with the same frame as~$t$. 
Remove the box containing the top-leftmost element~$x$ in~$t$, and perform the super jeu de taquin procedure~\eqref{SS:SuperJeuDeTaquin} on the resulting super skew tableau. 
We obtain a super tableau, denoted by~$t^\ast$, whose frame has one box removed from the one of $t$. Put the letter $x^{\ast}$ in the initial empty Young diagram in the same place as the box that was removed from the frame of~$t$. Repeat the algorithm on~$t^\ast$ and continue until all the elements of~$t$ have been removed and the initial empty Young diagram has been filled with their corresponding letters in~$\Si^{\mathrm{op}}$. Output the resulting super tableau for~$t^{\mathrm{op}}$.
\end{Algorithm}

This procedure, called the \emph{super evacuation},  is the super analogue of the \emph{Schützenberger's evacuation procedure},~\cite{Schutzenberger63}. 
See~\cite[Subsection~4.2]{Hage2021Super} for explicit constructions and examples.

\begin{property}[\cite{Hage2021Super}]
For any~$t$ in~$\YoungT(\Si)$, the super tableau~$t^{\mathrm{op}}$  satisfies the following equivalence:
\begin{equation}
\label{P:SuperEvacuationPlactic2}
(\R(t))^{\ast}\sim_{\P(\Si^{\mathrm{op}})} \R(t^{\mathrm{op}}),
\end{equation}
and  the map~$t\mapsto  t^{\mathrm{op}}$ is an involution on~$\YoungT(\Si)$.
\end{property}

Let~$\Si$ and~$\Si'$ be signed alphabets and~$w= \begin{pmatrix}
 x_1&\ldots &x_k\\
 y_1&\ldots &y_k
\end{pmatrix}$
 a signed two-rowed array on~$\Si$ and~$\Si'$. We define $w^\ast:= \begin{pmatrix}
 x_k^\ast&\ldots &x_1^\ast\\
 y_k^\ast&\ldots &y_1^\ast
\end{pmatrix}$
 the signed two-rowed array on~$\Si^{\mathrm{op}}$ and~$(\Si')^{\mathrm{op}}$  whose first and second rows are the opposites of the ones of~$w$. In particular, by definition, we have~$(w^\ast)^{\text{inv}} = (w^{\text{inv}})^\ast$.

\begin{proposition}
Let~$\Si$ and~$\Si'$ be signed alphabets and~$w$
 a signed two-rowed array on~$\Si$ and~$\Si'$. If~$\RSK(w)=(\T(w),\Q(w))$ then~$\RSK(w^\ast)=(\T(w)^{\mathrm{op}},\Q(w)^{\mathrm{op}})$.
\end{proposition}

\begin{proof}
Consider~$w= \begin{pmatrix}
 x_1&\ldots &x_k\\
 y_1&\ldots &y_k
\end{pmatrix}$
a signed two-rowed array on~$\Si$ and~$\Si'$.
Following Subsection~\ref{SS:SuperplacticMonoid}, the following equivalence~$\R(\T(w))\sim_{\P(\Si)} y_1\ldots y_k$ holds. Then we obtain
\[
\R(\T(w^{\ast})) \sim_{\P(\Si^{\mathrm{op}})}y_{k}^{\ast}\ldots y_{1}^{\ast} = (y_1\ldots y_k)^\ast \overset{\eqref{P:SuperEvacuationPlactic1}}{\sim_{\P(\Si^{\mathrm{op}})}} (\R(\T(w)))^\ast \overset{\eqref{P:SuperEvacuationPlactic2}}{\sim_{\P(\Si^{\mathrm{op}})}} \R(\T(w)^{\mathrm{op}}).
\]
 We deduce by Property~\ref{P:CrossSectionProperty} that the equality~$\T(w^{\ast}) = \T(w)^{\mathrm{op}}$
holds in~$\YoungT((\Si')^{\mathrm{op}})$. Hence, we obtain~$\RSK(w^\ast)=(\T(w)^{\mathrm{op}}, \Q(w^\ast))$. Similarly, we show that 
 \[
\RSK((w^\ast)^{\text{inv}}) = \RSK((w^{\text{inv}})^\ast)=(\Q(w)^{\mathrm{op}}, \T(w^\ast)).
\]
 We deduce by Theorem~\ref{T:SymmetryRSK} that~$(\Q(w)^{\mathrm{op}}, \T(w^\ast)) = (\Q(w^\ast), \T(w)^{\mathrm{op}})$, showing the claim.
\end{proof}

\subsubsection{A dual construction of the super RSK-correspondence}
\label{SSS:DualConstructionRSK}
We  present a dual way to construct the super RSK-correspondence using the left insertion algorithm on super tableaux.

\begin{Algorithm}[\cite{LaScalaNardozzaSenato06}]
\label{A:LeftInsertion} 
The \emph{left (or column) insertion}, denoted by~$\insl{}$,  inserts an element~$x$ in~$\Si$ into a super tableau~$t$ of~$\YoungT(\Si)$ as follows:
~\par\nobreak
\emph{Input:} A super tableau $t$ and a letter $x \in \Si$.

\emph{Output:} A super tableau $x \insl{} t$.

\emph{Method:} If $t$ is empty, create a box and label it $x$. Suppose~$t$ is 
non-empty. 
If~$x\in\Si_0$ (resp.~$x\in~\Si_1$) is larger than (resp. at least as large as) the bottom element of the
leftmost column of~$t$, then put~$x$ in a box to the bottom of this column; Otherwise, let~$y$ be the smallest element of the leftmost column of~$t$ such that~$y\geq x$ (resp.~$y>x$). Then replace~$y$ by~$x$ in this column  and recursively insert~$y$ into the super tableau formed by the columns of~$t$ to the right of the leftmost. Note that this recursion may end with an insertion into an empty column to the right of the existing columns of~$t$.
Output the resulting super tableau.
\end{Algorithm}

\begin{property}
\label{P:CommutationProperty}
As a consequence of Property~\ref{P:CrossSectionProperty}, the following \emph{commutation property} holds in~$\YoungT(\Si)$, for all~$t$ in~$\YoungT(\Si)$ and $x,y$ in~$\Si$:
\[
y \insl{} (t \insr{} x)
\: = \:
(y \insl{} t) \insr{} x.
\]
In particular, for any word~$w=x_1\ldots x_k$ in~$\Si^\ast$, the super tableau~$\T(w)$ is also computed by inserting its elements iteratively from right to left using the left insertion~$\insl{}$ as follows:
\[
\T(w)
\; =\;
(w\insl{} \emptyset)
\; :=\;
(x_1 \insl{} ( \ldots\insl{}  (x_k\insl{} \emptyset)\ldots)).
\]
\end{property}

\begin{Algorithm}
\label{A:RSkalgoDual}
Let~$\Si$ and~$\Si'$ be  signed alphabets. Starting from a signed two-rowed array~$w$ on~$\Si$ and~$\Si'$, we can compute the pair of super tableaux~$\RSK(w)=(\T(w),\Q(w))$ using the left insertion, as follows:
~\par\nobreak
\emph{Input:} A signed two-rowed array~$w= \begin{pmatrix}
 x_1&\ldots &x_k\\
 y_1&\ldots &y_k
\end{pmatrix}$
on $\Si$ and~$\Si'$.

\emph{Output:} A pair~$(\T'(w), \Q'(w))\in \YoungT(\Si')\times \YoungT(\Si)$ of same-shape tableaux containing $k$ boxes.

\emph{Method:} Start with an empty super tableau~$T'_{k+1}$ and an empty  super tableau~$Q'_{k+1}$. For each $i = k,\ldots, 1$, compute~$ y_i\insl{} T'_{i +1}$ as per Algorithm~\ref{A:LeftInsertion} and let $T'_i$ be the resulting super tableau.
Let~$Q'_i$ be super tableau obtained from~$Q'_{i+1}$, by performing the reverse sliding algorithm~\eqref{SSS:ReverseSlidings}, using the box that belongs to~$T'_i$ but not to~$T'_{i+1}$, and then place~$x_i$ in the top-leftmost corner of the result.
Output $T'_1$ for $\T'(w)$ and $Q'_1$ as $\Q'(w)$.
\end{Algorithm}

\begin{proposition}
\label{P:RSkalgoDual}
The output of Algorithm~\ref{A:RSkalgoDual} is equal to $\RSK(w)=(\T(w),\Q(w))$.
\end{proposition}

\begin{proof}
Following Property~\ref{P:CommutationProperty}, the super tableaux~$\T(w)$ and~$\T'(w)$ are equal. We still have to show that~$\Q(w) = \Q'(w)$. We will proceed by induction on the number of couples in~$w$. The result is obvious when~$w$ contains zero or one couple. 
Let~$w'$ be the signed two-rowed array obtained from~$w$ by eliminating its first couple~$(x_1,y_1)$.
By the induction hypothesis, we have~$\Q(w') = \Q'(w')$.
Moreover, on the one hand, and since the super tableau~$\Q'(w)$ is computed using the reverse sliding algorithm, we have that~$\Q'(w') = (\Q'(w))^\ast$, with the property that~$\Q'(w)$  has~$x_1$ in its top-leftmost corner.  On the second hand, by applying Proposition~\ref{P:SuperRSKSuperJeuDeTaquin} to the signed two-rowed array~$w'$ and  the super tableau containing one box filled with~$y_1$, we obtain that~$\Q(w') = (\Q(w))^\ast$, and then~$\Q'(w') = (\Q(w))^\ast$, with the property  that~$\Q(w)$  has~$x_1$ in its top-leftmost corner, showing that~$\Q(w) = \Q'(w)$.
\end{proof}

\begin{example}
Consider~$w=\setcounter{MaxMatrixCols}{20}
\begin{pmatrix}
  1&2&2&2&3&3&3&4&4&4&5&5&5&6&6&6\\
 2&4&2&1&4&5&6&3&5&6&4&3&1&5&4&1
\end{pmatrix}$ of Example~\ref{Ex:RSKExample1}.
The sequence of pairs produced during the computation of~$(\T(w),\Q(w))$ as per Algorithm~\ref{A:RSkalgoDual}  is the following:
\[
\big(\emptyset
\;
,
\;
\emptyset\big),
\quad
\big(
{\ytableausetup{mathmode, boxsize=1em}
\begin{ytableau}
1
\end{ytableau}}
\;
,
\;
{\ytableausetup{mathmode, boxsize=1em}
\begin{ytableau}
6
\end{ytableau}}
\big),
\quad
\big(
{\ytableausetup{mathmode, boxsize=1em}
\begin{ytableau}
1\\
4
\end{ytableau}}
\;
,
\;
{\ytableausetup{mathmode, boxsize=1em}
\begin{ytableau}
6\\
6
\end{ytableau}}
\big),
\quad
\Big(\raisebox{0.1cm}{
{\ytableausetup{mathmode, boxsize=1em}
\begin{ytableau}
1\\
4\\
5
\end{ytableau}}}
\;
,
\;
\raisebox{0.1cm}{
{\ytableausetup{mathmode, boxsize=1em}
\begin{ytableau}
6\\
6\\
6
\end{ytableau}
}}\Big),
\quad
\bigg(\raisebox{0.25cm}{
{\ytableausetup{mathmode, boxsize=1em}
\begin{ytableau}
1&1\\
4\\
5
\end{ytableau}}}
\;
,
\;
\raisebox{0.25cm}{
{\ytableausetup{mathmode, boxsize=1em}
\begin{ytableau}
5&6\\
6\\
6
\end{ytableau}}
}\bigg),
\quad
\bigg(\raisebox{0.25cm}{
{\ytableausetup{mathmode, boxsize=1em}
\begin{ytableau}
1&1\\
3&4\\
5
\end{ytableau}}}
\;
,
\;
\raisebox{0.25cm}{
{\ytableausetup{mathmode, boxsize=1em}
\begin{ytableau}
5&6\\
5&6\\
6
\end{ytableau}}
}\bigg),
\]
\[
\bigg(\raisebox{0.25cm}{
{\ytableausetup{mathmode, boxsize=1em}
\begin{ytableau}
1&1\\
3&4\\
4&5
\end{ytableau}}}
\;
,
\;
\raisebox{0.25cm}{
{\ytableausetup{mathmode, boxsize=1em}
\begin{ytableau}
5&6\\
5&6\\
5&6
\end{ytableau}}
}\bigg),
\quad
\bigg(\raisebox{0.25cm}{
{\ytableausetup{mathmode, boxsize=1em}
\begin{ytableau}
1&1\\
3&4\\
4&5\\
6
\end{ytableau}}}
\;
,
\;
\raisebox{0.25cm}{
{\ytableausetup{mathmode, boxsize=1em}
\begin{ytableau}
4&6\\
5&6\\
5&6\\
5
\end{ytableau}}
}\bigg),
\quad
\Bigg(\raisebox{0.25cm}{
{\ytableausetup{mathmode, boxsize=1em}
\begin{ytableau}
1&1\\
3&4\\
4&5\\
5&6
\end{ytableau}}}
\;
,
\;
\raisebox{0.25cm}{
{\ytableausetup{mathmode, boxsize=1em}
\begin{ytableau}
4&4\\
5&6\\
5&6\\
5&6
\end{ytableau}}
}\Bigg),
\quad
\Bigg(\raisebox{0.25cm}{
{\ytableausetup{mathmode, boxsize=1em}
\begin{ytableau}
1&1&4\\
3&3\\
4&5\\
5&6
\end{ytableau}}}
\;
,
\;
\raisebox{0.25cm}{
{\ytableausetup{mathmode, boxsize=1em}
\begin{ytableau}
4&4&4\\
5&6\\
5&6\\
5&6
\end{ytableau}}
}\Bigg),
\]
\[
\Bigg(\raisebox{0.25cm}{
{\ytableausetup{mathmode, boxsize=1em}
\begin{ytableau}
1&1&4\\
3&3\\
4&5\\
5&6\\
6
\end{ytableau}}}
\;
,
\;
\raisebox{0.25cm}{
{\ytableausetup{mathmode, boxsize=1em}
\begin{ytableau}
3&4&4\\
4&6\\
5&6\\
5&6\\
5
\end{ytableau}}
}\Bigg),
\quad
\Bigg(\raisebox{0.25cm}{
{\ytableausetup{mathmode, boxsize=1em}
\begin{ytableau}
1&1&4\\
3&3\\
4&5\\
5&6\\
5&6
\end{ytableau}}}
\;
,
\;
\raisebox{0.25cm}{
{\ytableausetup{mathmode, boxsize=1em}
\begin{ytableau}
3&3&4\\
4&4\\
5&6\\
5&6\\
5&6
\end{ytableau}}
}\Bigg),
\quad
\Bigg(\raisebox{0.3cm}{
{\ytableausetup{mathmode, boxsize=1em}
\begin{ytableau}
1&1&4\\
3&3&5\\
4&4\\
5&6\\
5&6
\end{ytableau}}}
\;
,
\;
\raisebox{0.3cm}{
{\ytableausetup{mathmode, boxsize=1em}
\begin{ytableau}
3&3&3\\
4&4&4\\
5&6\\
5&6\\
5&6
\end{ytableau}}
}\Bigg),
\]
\[
\Bigg(\raisebox{0.3cm}{
{\ytableausetup{mathmode, boxsize=1em}
\begin{ytableau}
1&1&1&4\\
3&3&5\\
4&4\\
5&6\\
5&6
\end{ytableau}}}
\;
,
\;
\raisebox{0.3cm}{
{\ytableausetup{mathmode, boxsize=1em}
\begin{ytableau}
2&3&3&3\\
4&4&4\\
5&6\\
5&6\\
5&6
\end{ytableau}}
}\Bigg),
\quad
\Bigg(\raisebox{0.5cm}{
{\ytableausetup{mathmode, boxsize=1em}
\begin{ytableau}
1&1&1&4\\
2&3&3&5\\
4&4\\
5&6\\
5&6
\end{ytableau}}}
\;
,
\;
\raisebox{0.5cm}{
{\ytableausetup{mathmode, boxsize=1em}
\begin{ytableau}
2&3&3&3\\
2&4&4&4\\
5&6\\
5&6\\
5&6
\end{ytableau}}
}\Bigg),
\quad
\Bigg(\raisebox{0.5cm}{
{\ytableausetup{mathmode, boxsize=1em}
\begin{ytableau}
1&1&1&4\\
2&3&3&5\\
4&4&4\\
5&6\\
5&6
\end{ytableau}}}
\;
,
\;
\raisebox{0.5cm}{
{\ytableausetup{mathmode, boxsize=1em}
\begin{ytableau}
2&3&3&3\\
2&4&4&4\\
2&5&6\\
5&6\\
5&6
\end{ytableau}}
}\Bigg),
\]
\[
\Bigg(\raisebox{0.5cm}{
\raisebox{-0.5cm}{$\T(w)\; =\; $}
{\ytableausetup{mathmode, boxsize=1em}
\begin{ytableau}
1&1&1&4&4\\
2&3&3&5\\
2&4&4\\
5&6\\
5&6
\end{ytableau}}}
\;
,
\;
\raisebox{0.5cm}{
\raisebox{-0.5cm}{$\Q(w)\; =\; $}
{\ytableausetup{mathmode, boxsize=1em}
\begin{ytableau}
1&2&3&3&3\\
2&4&4&4\\
2&5&6\\
5&6\\
5&6
\end{ytableau}}
}\Bigg).
\]

Consider now~$w=\setcounter{MaxMatrixCols}{20}
\begin{pmatrix}
 1&1&2&2&3&3&4&4\\
 3&2&1&2&4&3&1&2
\end{pmatrix}$ of Example~\ref{Ex:RSKExample2}.
The sequence of pairs produced during the computation of~$\T(w)$ and~$\Q(w)$ as per Algorithm~\ref{A:RSkalgoDual} starting from~$w$ is the following:
\[
\big(\emptyset
\;
,
\;
\emptyset\big),
\;
\big(
{\ytableausetup{mathmode, boxsize=1em}
\begin{ytableau}
2
\end{ytableau}}
\;
,
\;
{\ytableausetup{mathmode, boxsize=1em}
\begin{ytableau}
4
\end{ytableau}}
\big),
\;
\big(
{\ytableausetup{mathmode, boxsize=1em}
\begin{ytableau}
1&2
\end{ytableau}}
\;
,
\;
{\ytableausetup{mathmode, boxsize=1em}
\begin{ytableau}
4&4
\end{ytableau}}
\big),
\;
\Big(\raisebox{0.1cm}{
{\ytableausetup{mathmode, boxsize=1em}
\begin{ytableau}
1&2\\
3
\end{ytableau}}}
\;
,
\;
\raisebox{0.1cm}{
{\ytableausetup{mathmode, boxsize=1em}
\begin{ytableau}
3&4\\
4
\end{ytableau}
}}\Big),
\;
\bigg(\raisebox{0.25cm}{
{\ytableausetup{mathmode, boxsize=1em}
\begin{ytableau}
1&2\\
3\\
4
\end{ytableau}}}
\;
,
\;
\raisebox{0.25cm}{
{\ytableausetup{mathmode, boxsize=1em}
\begin{ytableau}
3&4\\
3\\
4
\end{ytableau}}
}\bigg),
\;
\bigg(\raisebox{0.25cm}{
{\ytableausetup{mathmode, boxsize=1em}
\begin{ytableau}
1&2\\
2&3\\
4
\end{ytableau}}}
\;
,
\;
\raisebox{0.25cm}{
{\ytableausetup{mathmode, boxsize=1em}
\begin{ytableau}
2&3\\
3&4\\
4
\end{ytableau}}
}\bigg),
\]
\[
\bigg(\raisebox{0.25cm}{
{\ytableausetup{mathmode, boxsize=1em}
\begin{ytableau}
1&1&2\\
2&3\\
4
\end{ytableau}}}
\;
,
\;
\raisebox{0.25cm}{
{\ytableausetup{mathmode, boxsize=1em}
\begin{ytableau}
2&2&3\\
3&4\\
4
\end{ytableau}}
}\bigg),
\quad
\bigg(\raisebox{0.25cm}{
{\ytableausetup{mathmode, boxsize=1em}
\begin{ytableau}
1&1&2\\
2&3\\
2&4
\end{ytableau}}}
\;
,
\;
\raisebox{0.25cm}{
{\ytableausetup{mathmode, boxsize=1em}
\begin{ytableau}
1&2&3\\
2&3\\
4&4
\end{ytableau}}
}\bigg),
\quad
\bigg(
\raisebox{0.25cm}{
\raisebox{-0.2cm}{$\T(w)\; =\; $}
{\ytableausetup{mathmode, boxsize=1em}
\begin{ytableau}
1&1&2\\
2&3\\
2&4\\
3
\end{ytableau}}}
\;
,
\;
\raisebox{0.25cm}{
\raisebox{-0.2cm}{$\Q(w)\; =\; $}
{\ytableausetup{mathmode, boxsize=1em}
\begin{ytableau}
1&2&3\\
1&3\\
2&4\\
4
\end{ytableau}}
}\bigg).
\]
\end{example}

\section{The super Littlewood--Richardson rule}
\label{S:SuperLittelwoodRichardsonRule}

In this section, we apply the super-RSK correspondence in order to give a combinatorial version of the super Littlewood--Richardson rule on super Schur functions over a finite signed alphabet.

In the sequel, we will assume that~$\Si$ is finite.

\subsection{Super Littlewood--Richardson coefficients}
\label{SS:SuperLittlewoodRichardsonCoefficients}
Let~$\lambda$,~$\mu$ and~$\nu$ be in~$\Pr$ such that~$\nu/\lambda$ is a skew shape.
We want to compute the number of ways a given super tableau~$t$ in~$\YoungT(\Si,\nu)$ can be written as the product of a super tableau~$t'$ in~$\YoungT(\Si,\lambda)$ and a super tableau~$t''$ in~$\YoungT(\Si,\mu)$.

For any super tableau~$t$ in~$\YoungT(\Si,\nu)$, we set
\[
\YoungT(\Si,\lambda, \mu,  \dashv t)\; :=\; \big\{ (t',t'')\in \YoungT(\Si,\lambda)\times \YoungT(\Si,\mu)\;\big|\; t= t'\star_{\YoungT(\Si)}t'' \big\},
\]
and we call the integer
\[
c_{\lambda,\mu}^{\nu} \; := \; \# \YoungT(\Si,\lambda, \mu,  \dashv t)
\]
the \emph{super Littlewood--Richardson coefficient}.

Using the super jeu de taquin and as discussed in~\ref{SSS:InsertionsSuperJeuDeTaquin}, the integer~$c_{\lambda,\mu}^{\nu}$ is equal to the number of super skew tableaux of the following form:
\[
\raisebox{1.4cm}{$[t',t'']\;=\;$}
\scalebox{0.6}{
\begin{tikzpicture}
\draw  (0,0) --++(0,2) -- ++(2,0)-- ++(0,-0.75)-- ++(-1,0)-- ++(0,-1)-- ++(-1,0)-- ++(-2,0)-- ++(0,-2)-- ++(1,0)-- ++(0,1)-- ++(1,0)-- ++(0,0.75);
\draw (0.2,1.2) node[right]{\scalebox{1.5} {$t''$}} ;
\draw (-1.75,-0.75) node[right]{\scalebox{1.5} {$t'$}};
\end{tikzpicture}}
\]
whose rectification is~$t$.
We will  prove  that the integer~$c_{\lambda,\mu}^{\nu}$ is also  equal to the number of super skew tableaux in~$\SkewT(\Si,\nu/\lambda)$  whose rectification is a given super tableau of shape~$\mu$.

For any~$T$ in~$\YoungT(\Si,\mu)$, we set 
\[
\SkewT(\Si,\nu/\lambda, \dashv T)\; :=\; \big\{S\in \SkewT(\Si,\nu/\lambda)\;\big|\; \Rec(S)=T\big\}.
\]
We prove in Theorem~\ref{T:SuperLittlewoodRichardsonRule} that there is a canonical one-to-one correspondence between the sets~$\YoungT(\Si,\lambda, \mu,  \dashv t)$ and~$\SkewT(\Si,\nu/\lambda, \dashv T)$. We deduce that the coefficient~$c_{\lambda,\mu}^{\nu}$  does not depend on~$t$ and~$T$, and depends only on~$\lambda$,~$\mu$  and~$\nu$. Note that the later property is proved in~\cite{Hage2021Super} using an interpretation of the super jeu de taquin in terms of Fomin's growth diagrams.

\subsection{Super Littlewood--Richardson rule}
\label{SS:SuperLittlewoodRichardsonRule}

We denote by~$R_{\Si}$ the~$\mathbb{Z}$-algebra constructed from the super plactic monoid~$\P(\Si)$ whose linear generators are the monomials in~$\P(\Si)$.
This algebra is an associative and unitary ring that is not commutative. A generic element in~$R_{\Si}$ is realized by a formal sum of super  plactic classes with coefficients from~$\mathbb{Z}$. 
Following Property~\ref{P:CrossSectionProperty}, a typical element in~$R_{\Si}$ is a formal sum of super tableaux.
A canonical homomorphism from $R_{\Si}$ onto the ring of polynomials~$\mathbb{Z}[X]$ is obtained by taking each super tableau $t$ to its monomial~$x^t$, where~$x^t$ is the product of the variables~$x_i$, each occurring as many times in $x^t$ as $i$ occurs in $t$. For instance, the following monomial
\[
x^{\T(w)} = x_{1}^{3}x_{2}^{2}x_{3}^{2}
x_{4}^{4}x_{5}^{3}x_{6}^{2}\quad (\text{resp}.~x^{\Q(w)} = x_{1}^{1}x_{2}^{3}x_{3}^{3}
x_{4}^{3}x_{5}^{3}x_{6}^{3}) \]
corresponds to~$\T(w)$ (resp.~$\Q(w)$) computed in Example~\ref{Ex:RSKExample1}. Moreover, the following monomial
\[
x^{\T(w)} = x_{1}^{2}x_{2}^{2}x_{3}^{2}
x_{4}\quad (\text{resp}.~x^{\Q(w)} = x_{1}^{2}x_{2}^{2}x_{3}^{2}
x_{4}^{2}) \]
corresponds to the super tableau~$\T(w)$ (resp.~$\Q(w)$) computed in Example~\ref{Ex:RSKExample2}.

We define~$S_\lambda$ (resp.~$S_{\lambda/\mu}$) in~$R_{\Si}$ to be the sum of all super tableaux (resp. super skew tableaux) of shape~$\lambda$ (resp.~$\lambda/\mu$) and entries in~$\Si$, with~$\lambda\in\Pr$ (resp.~$\lambda/\mu$ is a skew shape). 
By taking the image of~$S_{\lambda}$ (resp.~$S_{\lambda}/\mu$) in~$\mathbb{Z}[X]$, we obtain the so-called \emph{super Schur function} (resp. \emph{super skew Schur function})~$s_{\lambda}(X)$ (resp.~$s_{\lambda/\mu}(X)$).

\begin{theorem}[The super Littlewood--Richardson rule]
\label{T:SuperLittlewoodRichardsonRule}
Let~$\lambda$,~$\mu$ and~$\nu$ be partitions in~$\Pr$ such that~$\nu/\lambda$ is a skew shape. 
The following identities
\begin{equation}
\label{Eq:SuperLittlewoodRichardsonRule}
S_{\lambda} S_{\mu}\;=\; \underset{\nu}{\sum} c_{\lambda,\mu}^{\nu} S_{\nu}
\qquad
\text{ and }
\qquad
S_{\nu/\lambda}\;=\; \underset{\mu}{\sum} c_{\lambda,\mu}^{\nu} S_{\mu}
\end{equation}
hold in~$R_{\Si}$. 
\end{theorem}

\begin{proof}
Prove that for any~$t$ in~$\YoungT(\Si,\nu)$ and any~$T$ in~$\YoungT(\Si,\mu)$, there is a canonical one-to-one correspondence between the sets~$\YoungT(\Si,\lambda, \mu,  \dashv t)$ and~$\SkewT(\Si,\nu/\lambda, \dashv T)$.
Start with~$(t',t'')$ in~$\YoungT(\Si,\lambda, \mu,  \dashv t)$ and let
\[
\begin{pmatrix}
 x_1&\ldots &x_k\\
 y_1&\ldots &y_k
\end{pmatrix}
\]
be the signed two-rowed array on~$\Si$ corresponding to the couple~$(t'', T)$.
By inserting the letters~$y_i$, for~$i=1,\ldots, k$,  into the super tableau~$t'$, we form a super skew tableau~$S$ by successively placing the letter~$x_i$, for~$i=1,\ldots, k$, into the new boxes starting with an empty Young diagram of the same shape as~$t'$.
Since the super tableau~$t=t'\star_{\YoungT(\Si)}t''$ has shape~$\nu$, then the super skew tableau~$S$ has shape~$\nu/\lambda$ and we deduce by Proposition~\ref{P:SuperRSKSuperJeuDeTaquin} that~$S$ belongs to~$\SkewT(\Si,\nu/\lambda, \dashv T)$.

Conversely, start with a super skew tableau~$S$ in~$\SkewT(\Si,\nu/\lambda, \dashv T)$, and
let~$T'$ be in~$\YoungT(\Si^{<},\lambda)$ where~$\Si^{<}$ is a signed alphabet whose elements are all smaller than the ones of~$\Si$.
Let~$T''$ be the super tableau in~$\YoungT(\Si^{<}\cup\Si,\nu)$ obtained by placing~$T'$ on the empty part of~$S$ corresponding to~$\nu$.  By Theorem~\ref{T:SuperRSKCorrespondence}, the couple~$(t, T'')$ corresponds to a unique signed two-rowed array of the following form~$\begin{pmatrix}
x'_1&\ldots &x'_l& x_1&\ldots &x_k\\
y'_1&\ldots &y'_l& y_1&\ldots &y_k
\end{pmatrix}$.
Then the signed two rowed array~$
\begin{pmatrix}
x'_1&\ldots &x'_l\\
y'_1&\ldots &y'_l
\end{pmatrix}$ corresponds to the couple~$(t',T')$, and following Proposition~\ref{P:SuperRSKSuperJeuDeTaquin} the signed two-rowed array~$\begin{pmatrix}
 x_1&\ldots &x_k\\
y_1&\ldots &y_k
\end{pmatrix}$ corresponds to the couple~$(t'', T)$, for some super tableaux~$t'$ in~$\YoungT(\Si,\lambda)$ and~$t''$ in~$\YoungT(\Si,\mu)$,  such that~$t'\star_{\YoungT(\Si)}t'' = t$. Hence, we obtain a pair~$(t', t'')$ in~$\YoungT(\Si,\lambda, \mu,  \dashv t)$. We deduce that the cardinal number of the set~$\SkewT(\Si,\nu/\lambda, \dashv T)$  is equal to  the coefficient~$c_{\lambda,\mu}^{\nu}$, showing the claim.
\end{proof}

Since neither~$\YoungT(\Si,\lambda, \mu,  \dashv t)$ nor~$\SkewT(\Si,\nu/\lambda, \dashv T)$ in the correspondence of the proof of Theorem~\ref{T:SuperLittlewoodRichardsonRule} depends on the contents of the super tableaux used to define the other, we deduce the following result:

\begin{corollary}
\label{C:SuperLittlewoodRichardsonRule}
The cardinal number~$c_{\lambda,\mu}^{\nu}$ of the sets~$\YoungT(\Si,\lambda, \mu,  \dashv t)$ and~$\SkewT(\Si,\nu/\lambda, \dashv T)$ is independent of choice of~$t$ and~$T$ and depends only on~$\lambda$,~$\mu$  and~$\nu$.
\end{corollary}

\subsubsection{Super Littlewood--Richardson rule and Yamanouchi tableaux}
\label{SS:Yamanouchi}
A word over $[n]$ is called \emph{Yamanouchi} if, when read from left to right, the number of occurrences of $i'$ is always greater than or equal to the number of occurrences of $(i+1)'$ for all $i \ge 1$.
A skew tableau over $[n]$ is said to have \emph{weight} $(\lambda_1,\ldots,\lambda_\ell)$ if it contains exactly $\lambda_1$ entries equal to $1'$, $\lambda_2$ entries equal to $2'$, and so on up to $\lambda_\ell$ entries equal to $\ell'$.
A skew tableau $S$ over $[n]$ is called a \emph{Littlewood--Richardson tableau} if its reading word $\R(S)$ is Yamanouchi. In~\cite{Hage25}, we proved that there exists a bijection between pairs of tableaux in
\[
\displaystyle \LR(\mu, \lambda)\times \YoungT(\Si,\nu)
\]
and ordered pairs of tableaux in
\[
\displaystyle\YoungT(\Si,\lambda)\times\YoungT(\Si,\mu),
\]
where $\LR(\mu,\lambda)$ denotes the set of all Littlewood--Richardson tableaux of shape $\nu/\lambda$ and weight~$\mu$, for a partition $\nu$ containing $\lambda$.
As a consequence, we showed that for partitions~$\lambda$, $\mu$, and $\nu$ with~$\lambda\subseteq \nu$, the Littlewood--Richardson coefficient $c_{\lambda,\mu}^{\nu}$ is equal to the number of Littlewood--Richardson tableaux of shape $\nu/\lambda$ and weight $\mu$. Further details and explicit constructions are given in~\cite{Hage25}.

\begin{example}
\label{Ex:SuperLRExample}
Consider~$\Si=[4]$ with signature given by~$\Si_0$ the set of even numbers and~$\Si_1$  defined consequently. 
Let~$\lambda=(3,1,1)$ and~$\mu=(2,1)$, then~$\LR(\mu,\lambda)$ consists of the following:
\[
\ytableausetup{mathmode, boxsize=1em}
\begin{ytableau}
\empty&\empty&\empty&1&1\\
\empty&2\\
\empty&\none
\end{ytableau}
\qquad
\ytableausetup{mathmode, boxsize=1em}
\begin{ytableau}
\empty&\empty&\empty&1&1\\
\empty&\none\\
\empty&\none\\
2
\end{ytableau}
\qquad
\ytableausetup{mathmode, boxsize=1em}
\begin{ytableau}
\empty&\empty&\empty&1\\
\empty&1&2\\
\empty&\none
\end{ytableau}
\qquad
\ytableausetup{mathmode, boxsize=1em}
\begin{ytableau}
\empty&\empty&\empty&1\\
\empty&1\\
\empty&\none\\
2
\end{ytableau}
\qquad
\ytableausetup{mathmode, boxsize=1em}
\begin{ytableau}
\empty&\empty&\empty&\none\\
\empty&1&1\\
\empty&2
\end{ytableau}
\qquad
\ytableausetup{mathmode, boxsize=1em}
\begin{ytableau}
\empty&\empty&\empty&\none\\
\empty&1&1\\
\empty&\none\\
2
\end{ytableau}
\]
\[
\ytableausetup{mathmode, boxsize=1em}
\begin{ytableau}
\empty&\empty&\empty&1\\
\empty&1\\
\empty&2
\end{ytableau}
\qquad
\ytableausetup{mathmode, boxsize=1em}
\begin{ytableau}
\empty&\empty&\empty&1\\
\empty&\none\\
\empty&\none\\
1\\
2
\end{ytableau}
\qquad
\ytableausetup{mathmode, boxsize=1em}
\begin{ytableau}
\empty&\empty&\empty&\none\\
\empty&1\\
\empty&\none\\
1\\
2
\end{ytableau}
\qquad
\ytableausetup{mathmode, boxsize=1em}
\begin{ytableau}
\empty&\empty&\empty&\none\\
\empty&1\\
\empty&2\\
1
\end{ytableau}
\qquad
\ytableausetup{mathmode, boxsize=1em}
\begin{ytableau}
\empty&\empty&\empty&1\\
\empty&2\\
\empty&\none\\
1
\end{ytableau}.
\]
Hence, we have the following decomposition in~$R_{\Si}$:
\[
\begin{array}{rl}
S_{(3,1,1)}S_{(2,1)}=&S_{(4,3,1)}+S_{(3,3,2)}+S_{(3,3,1,1)}+S_{(5,2,1)}+S_{(5,1,1,1)}\\
                               &+2S_{(4,2,1,1)}+S_{(4,2,2)}+S_{(4,1,1,1,1)}+S_{(3,2,1,1,1)}+S_{(3,2,2,1)}.
\end{array}
\]
\end{example}

\subsubsection{Matrix version of the super-RSK correspondence}
\label{SS:MatrixVersion}

In~\cite{Hage26}, we gave bijective proofs of \emph{super Cauchy identities} by interpreting signed two-rowed arrays as matrices whose entries record the multiplicities of pairs drawn from signed alphabets. The parity of the indices imposes entrywise restrictions: pairs of mixed parity may occur at most once, whereas pairs of compatible parity may occur with arbitrary multiplicity. Via the super-RSK correspondence, these matrices are related to pairs of super tableaux of the same shape, with matrix transposition reflecting the symmetry of the correspondence. This framework leads to a combinatorial proof of a super Cauchy identity, expressing a rational generating function as a sum of products of super Schur functions. A variant of the construction, restricted to matrices with entries in~$\{0,1\}$, yields a correspondence between pairs of super tableaux of conjugate shapes and leads to a second super Cauchy identity. These results may be viewed as extensions of the classical Cauchy identities,~\cite{Knuth70}, and, under suitable order assumptions on the signed alphabets, recover earlier identities of Berele and Remmel,~\cite{BereleRemmel85}.

\begin{small}
\renewcommand{\refname}{\Large\textsc{References}}
\bibliographystyle{plain}
\bibliography{biblioCURRENT}
\end{small}

\quad

\vfill
\begin{flushright}
\begin{small}
\noindent \textsc{Nohra Hage} \\
\url{nohra.hage@univ-catholille.fr} \\
ICL, Junia, Université Catholique de Lille,\\
 LITL, F-59000 Lille, France\\
\end{small}
\end{flushright}

\vspace{0.25cm}

\begin{small}---\;\;\today\;\;-\;\;\hhmm\;\;---\end{small} \hfill

\end{document}